\documentclass[10pt,reqno]{amsart}
  \usepackage[margin=1.1in]{geometry}
\usepackage{amsfonts}
\usepackage{amsmath}
\usepackage{amssymb}
\usepackage{amsthm}
\usepackage{amscd}
\usepackage{ulem}
\usepackage{graphicx}
\usepackage{url}
\usepackage[all]{xy}
\usepackage{amsfonts}
\usepackage[utf8]{inputenc}
\usepackage{MnSymbol}

\def\R{\mathbb{R}}
\def\Q{\mathbb{Q}}
\def\P{\mathbb{P}}
\def\C{\mathbb{C}}
\def\A{\mathbb{A}}
\def\Z{\mathbb{Z}}
\def\l{\langle}
\def\r{\rangle}
\def\Cs{\mathbb{C}^\times}
\def\x{\times}
\def\ox{\otimes}
\def\a{\alpha}
\def\la{\lambda}
\def\ga{\gamma}
\def\p{\partial}

\def\Tad{T_{ad}}
\def\oTad{\overline{T_{ad,0}}}
\def\Tsad{T^{\ltimes}_{ad}}
\def\oTsad{\overline{T^\ltimes_{ad,0}}}

\def\sLG{L^{sm}G}
\def\Taf{\tilde L^\ltimes T}

\def\tvt{ T^{\ltimes} \widetilde{\times} V_T}	
\def\tlt0{(\tilde LT/T)_0}

\def\cB{{\mathcal{B}}}
\def\cU{{\mathcal{U}}}
\def\mft{{\mathfrak{t}}}
\def\mfg{{\mathfrak{g}}}
\def\cG{\mathcal{G}}
\def\cO{\mathcal{O}}

\def\Ga{G^{aff}}
\def\Gap{G^{aff}_{poly}}
\def\Gas{G^{aff}_{poly}}
\def\LpG{L_{poly}G}
\def\Gaa{G^{aff}_{ad}}
\def\Gas{G^{aff}_{sm}}
\def\Ta{T^{aff}}
\def	\Ba{B^{aff}}
\def\Wa{W^{aff}}
\def\Gsd{L^{\ltimes}G}
\def\Gsdp{L^{\ltimes}_{poly}G}
\def\Tsd{T^{\ltimes}}
\def\Xa{X^{aff}}
\def\zG{G^\ltimes[z^\pm]}

\def\Xap{{X^{aff}_{poly}}}

\def\Xas{X^{aff,sm}}
\def\LTsd{L^\ltimes T}
\def\bO{\mathbf{O}}
\def\cXap{\mathcal{X}^{aff}_{poly}}

\DeclareMathOperator{\ec}{Spec }
\DeclareMathOperator{\im}{im}

\newcommand{\abcd}[4]{\left(\begin{array}{cc}
  #1 & #2 \\ 
  #3 & #4 \\ 
\end{array} \right)}
\newcommand{\oG}[1]{\overline{G_{#1}}}
\newcommand{\colv}[4]{\left[\begin{array}{c} #1 \\ #2\\ #3 \\ #4 \end{array} \right]}
\newcommand{\mc}[1]{\mathcal{#1}}
\newcommand{\ol}[1]{\overline{#1}}
\newcommand{\ch}[1]{\Lambda_{#1}}
\newcommand{\fq}[3]{\frac{#1}{2}#3 Q#3 + #2 Q #3}

\makeatletter
\newtheorem*{rep@theorem}{\rep@title}
\newcommand{\newreptheorem}[2]{%
\newenvironment{rep#1}[1]{%
 \def\rep@title{#2 \ref{##1}}%
 \begin{rep@theorem}}%
 {\end{rep@theorem}}}
\makeatother

\newreptheorem{theorem}{Theorem}

\newtheorem{thm}{Theorem}[section]
\newtheorem{prop}[thm]{Proposition}
\newtheorem{lemma}[thm]{Lemma}
\newtheorem{cor}[thm]{Corollary}

\theoremstyle{definition}
\newtheorem{definition}[thm]{Definition}
\theoremstyle{remark}
\newtheorem{rmk}{Remark}
 \usepackage[refpage]{nomencl}
\makenomenclature

 \title{A Wonderful Embedding of the Loop Group}
\author{ Pablo Solis}
\address{Department of Mathematics,
 University of California,
 Berkeley, CA}
\email{pablo@math.berkeley.edu}
\begin{document}

\begin{abstract}
I describe the wonderful compactification of loop groups. These compactifications are obtained by adding normal-crossing boundary divisors to the group $LG$ of loops in a reductive group G (or more accurately, to the semi-direct product $\Cs \ltimes LG$) in a manner equivariant for the left and right $\Cs \ltimes LG$-actions. The analogue for a torus group $T$ is the theory of toric varieties; for an adjoint group $G$, this is the wonderful compactications of De Concini and Procesi. The loop group analogue is suggested by work of Faltings in relation to the compactification of moduli of G-bundles over nodal curves. 
\end{abstract}


\maketitle
\tableofcontents

\section{Introduction}
Let $G$ be a simple simply connected algebraic group over $\C$. This paper studies smooth compactifications of $G$ and limits of families of principal $G$ bundles on nodal curves.  The connection between the two can be seen as follows.  If we fix a smooth curve $C$ over $\ec \C$ then the stack $Bun_G(C)$ of all $G$-bundles $C$ satisfies the {\it valuative criterion for completeness}, which means given a commutative square as below we can always fill in the dotted arrow:

\[
\xymatrix{ \ec K \ar[r]\ar[d] & Bun_G(C)\ar[d]\\
\ec R\ar@{-->}[ru]\ar[r] & \ec \C
}
\]
where $R$ is a complete DVR with fraction field $K$.
\bigskip

When $C$ is a nodal curve then $Bun_G(C)$ no longer has this property.  To see this let $E \to C$ be a principal $G$-bundle and $\pi \colon \tilde C \to C$ the normalization of $C$.  Then $E$ can be identified $\pi^* E$ together with an isomorphism $\phi \colon\pi^* E_y \to \pi^* E_z$ where $y,z$ are the pre images of the node $x\in C$.  We can consider $\phi\in G$ and in families $\phi$ can go to infinity.  Thus compactifications of $G$ are relevant to any completion of $Bun_G(C)$.

Over a fixed nodal curve Bhosle in \cite{Bhosle} has completed $Bun_G(C)$ simply by compactifying $G$.  However this does not address how to compactify $Bun_G$ over a family of curves and does not give a modular interpretation of what the boundary of the completion means.  Such a modular interpretation was given by Kausz for $G = GL(V)$.  
The key innovation we demonstrate here is that in order to provide a similar construction for $G$ simple one must not just compactifify $G$ but in fact `compactify' or complete the loop group $LG$.

Much work has been done in both the subject of compactifying reductive groups and the study of bundles on curves via loop groups.  Let us give a brief account of some of the relevant results in these areas.

\subsection{Compactifications of $G$}
In 1983 De Concini and Procesi studied the symmetric space $G/H$ where $G$ is a complex Lie group and $H$ is the fixed point set of an involution $\sigma$ of $G$; see \cite{DC&P}.  They constructed a ``wonderful'' compactification $\ol{G/H}$ of $G/H$; see definition \ref{def:wond}. 
After De Concini and Procesi's result the properties of their compactification were axiomatized and such varieties were called {\it wonderful}.  A paticular case is $G = G\x G/\Delta(G)$; when $G$ is of adjoint type this gives a wonderful compactification of $G$.  A construction of the wonderful compactification which we will exploit uses representation theory:
\begin{equation}\label{eq:Gwond}
\ol{G} = \ol{G \x G.[id]} \subset \P End(V(\la) )
\end{equation}
where $\la$ is a regular dominant weight.

In fact smooth compactifications $\ol{G}$ for all reductive groups $G$ exist \cite[6.2.4]{Brion} but in general lack certain combinatorial properties required to be wonderful.  Additionally, there is a so called canonical embedding of a semi simple group but this compactification is generally not smooth unless $Z(G) = 1$.  The only exceptions are when $G = Sp_{2n}(\C)$.  Several places in the literature mistakenly state that the only exception is $G = Sp_{2}(\C) = SL_2(\C)$.  I thank Johan Martens for preventing another mistake here.


The canonical embedding for semi simple $G$ has finite quotient singularities and the singularities can be resolved by working with Deligne-Mumford stacks. In \cite{Martens}, Martens and Thaddeus carry this out explicitly by constructing certain moduli problems about $G$-bundles on chains of $\P^1$s that {\it represent} the compactification. In this paper, we give a different proof using representation theory. Namely for a regular dominant weight $\la$ there is a quasiprojective variety $Y$ with an action of a torus $T$ such that the global quotient $\mathcal{X} = [Y/T]$ contains $G$ as a dense open subvariety.  Additionally, $\mc{X}$ contains a dense open substack $\mc{X}_0$ which is the closure of the open cell $U^-TU$ of $G$ and
\bigskip
 
\begin{reptheorem}{thm:stackGad}
\begin{itemize} Let $G$ be a semi simple group then there is a stacky compactification $\mathcal{X}$ of $G$ with an action of $G\x G$  with the following properties
\item[(a)] $\mathcal{X}$ is smooth and proper.
\item[(b)] $\mathcal{X} - \mathcal{X}_0$ is of pure codimension $1$ and we have an exact sequence
\[
0\to \Z^r \to Pic(\mathcal{X}) \to Z \to 0
\]
where $Z$ is a finite group and the subgroup $\Z^r$ is generated by the irreducible components of $\mathcal{X} - \mathcal{X}_0$.
\item[(c)] The boundary $\mathcal{X} - G$ consists of $r$ divisors $D_1,\dotsc,D_{r}$ with simple normal crossings and the closure of the $G \x G$-orbits are in bijective correspondence with subsets $I \subset [1,r]$ in such a way that to $I$ we associate $\cap_{i \in I} D_i$.
\item[(d)] Let $u_1,\dotsc,u_{r}$ be generators of the rays of the Weyl chamber and $M$ be the monoid they generate.  Any $G$ equivariant $\mathcal{X}' \to \mathcal{X}$ determines and is determined by a fan supported in the negative Weyl chamber whose lattice points lie in $M$.
\end{itemize}
\end{reptheorem}
Though this result is not new the fact that is can be proved using only representation theory will be important when we turn to the study of loop groups.
\subsection{Loop groups}
The algebraic loop group is $LG = G((z)) = G(\ec \C((z)))$ and the smooth loop group is $\sLG = C^\infty(S^1,G)$.  We state the main result for $LG$ but we have an analogous result for $\sLG$ (see remark \ref{rmk:loopGad}).   

In fact, the group of interest is a semi-direct product $\Cs \ltimes LG$; this means for $u\in \Cs$ we have $u \ga(z) u^{-1} = \ga(u z)$.  The group $\Ga$ appearing in \ref{thm:loopGad} is a central extension of  $\Cs \ltimes LG$.  The object $\Xa$ is an ind-stack constructed using a regular dominant weight $\la$ of $\Ga$.

\begin{reptheorem}{thm:loopGad}
If $G$ is a simple, connected, simply connected algebraic group over $\C$ with center $Z(G)$ then $\Cs \ltimes LG/Z(G)$ has a wonderful embedding $\Xa = \ol{\Cs \ltimes LG/Z(G)}$.  There is a dense open substack $\Xa_0$ which is the closure of the open cell in $\Cs \ltimes LG/Z(G)$ and
\begin{itemize}
\item[(a)] $\Xa$ is independent of $\lambda$. 
\item[(b)] $\Xa - \Xa_0$ is of pure codimension $1$.  It is a union of $r+1$ divisors that are Cartier and freely generate the Picard group.
\item[(c)] The boundary $\Xa - \Gsd/Z(G)$ consists of $r+1$ normal crossing divisors $D_0,\dotsc,D_r$ and the closure of the $\Ga \x \Ga$-orbits are in bijective correspondence with subsets $I \subset [0,r]$ in such a way that to $I$ we associate $\cap_{i \in I} D_i$.
\item[(d)] Any $G$ equivariant $X' \to \Xa$ determine and is determined by a Weyl equivariant morphism of toric varieties $\overline{T}' \to \ol{\Tsd_{ad,0}}$.
\end{itemize}
\end{reptheorem}
We use the word embedding because $\ol{\Cs\ltimes LG}$ does not satisfy the valuative criterion for completeness. However we do have a completeness result for the polynomial loop group $G(\C[z^\pm])$; see theorem \ref{thm:complete}.

The strategy for proving \ref{thm:loopGad} is to use the representation theory of $\Ga$. More precisely, take a highest weight representation $V(\la)$ of $\Ga$.  $V(\la)$ is an infinite dimensional vector space which is a direct sum of weight spaces $V_{\mu}$ for a maximal torus $T^{aff}$ of $\Ga$.

Now we consider
\[
\Xa = \ol{\Ga \x \Ga.[id]} \subset \P \bigg[ V( \la)\hat \otimes V(\la)^* \bigg]
\]
Where $ V(\la)\hat \otimes V(\la)^* := \prod_{\mu,\nu} \hom(V_{\mu},V_{\nu})$ and the product is over the weight spaces of the representation.

There is also a stacky extension analogous to theorem \ref{thm:stackGad}, see theorem \ref{thm:loopstackGad}.

\subsection{$LG$ and $G$-Bundles on Curves}
In this subsection we will focus on the smooth loop group because it will allow us to draw some intuitive pictures.

For any connected topological group $G$ and fixed Riemann Surface $C$ we have the moduli stack $Bun_{G,hol}(C)$ of holomorphic principal $G$-bundles on $C$.  Atiyah and Segal \cite{loop} prove there is an isomorphism 
\begin{thm}\label{dcchol}
\[Bun_{G,hol}(C) \cong L^{sm}_C G\backslash \sLG / L^{sm,+}G\]
\end{thm}
Here $L^{sm,+}G$ consists of boundary values of holomorphic function from a small disc $ \{z \in \C: |z| < 1\} \cong D_p \subset C$ and $L^{sm}_C G$ consists of boundary values of holomorphic function $C - \ol{D_p} \to G$.  


Theorem \ref{dcchol} has a useful reformulation. We can write $C = C^0 \cup A$ where $A$ is an annulus and $C^0$ is the disjoint union of a disc $D$ and a noncompact Riemann surface $C - \ol{D}$; see figure \ref{pic:dcc}. Then for $Y_A =  \sLG \x \sLG/Hol(A,G)$ we have
\begin{equation}\label{dcccyl}
Bun_{G,hol}(C) = \frac{Y_A}{Hol(C^0,G)}
\end{equation}
Theorem \ref{dcchol} corresponds to gluing trivial bundles on a small disc and on the complement of a small disc with a transition function in $L^{sm}G$.  Equation \eqref{dcccyl} corresponds to gluing the same trivial bundles by a pair of elements in $\sLG$ over a cylinder or annulus; in either case you produce a $G$-bundle on the same curve.

\begin{figure*}[htm]
\centering
\includegraphics[scale=0.35]{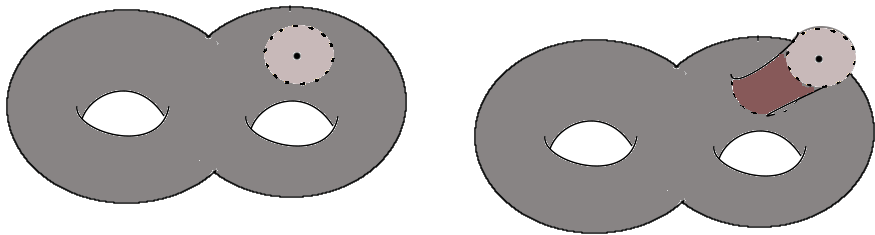}
\caption{The double coset construction.  The left picture is the analytic picture underlying theorem \ref{dcchol} and the right is the analytic picture underlying equation \eqref{dcccyl}.}
\label{pic:dcc}
\end{figure*}

With these preliminaries in place, let us describe an approach for studying $G$-bundles on a smooth curve degenerating to a nodal one. Holomorphically (or algebraically formally) any family of curves over a 1-dimensional base with smooth generic fiber and nodal special fiber has an \'{e}tale neighborhood that looks like the genus $0$ degeneration represented by the morphism
\[
\ec \frac{\C[x,y,u]}{xy -u} \to \ec \C[u].
\]
This is illustrated in figure \ref{pic:cyl}. For a small positive constant $\delta$, we can describe figure \ref{pic:cyl} as $A_u := \{ (x,y) \in \C^2 | |x|,|y|<\delta, xy = u\}$; a family of annuli parametrized by $u$.  
\begin{figure*}[htm]
\centering
\includegraphics[scale=0.35]{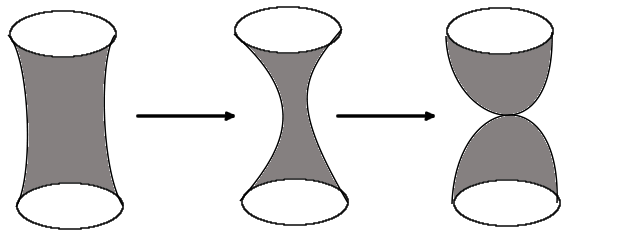}
\caption{The local picture of a nodal degeneration.}
\label{pic:cyl}
\end{figure*}
In other words, if $C_u$ is a family of smooth curves degenerating to a nodal curve $C_0$ then we can write $C_u = C_u^0 \cup A_u$.  In analogy with \eqref{dcccyl} we have
\begin{equation}\label{dccfam}
Bun_{G,hol}(C_u) = \frac{Y_{A_u}}{Hol(C^0_u,G)}
\end{equation}
where $Y_{A_u} = \sLG \x \sLG/Hol(A_u,G)$.  This construction allows us to reduce to studying the homogeneous space $Y_{A_u}$.  

Let us now ignore the norm restrictions on $x,y$ and allow and $u\in \C = \A^1$. We'll write this as $Y_A \xrightarrow{p} \A^1$ and for $u \in \A^1$ we have $p^{-1}(u) = Y_{A_u}$. 

The main idea can now be presented as follows.  One can show $p^{-1}(\Cs) \cong \Cs\ltimes LG$ and $p^{-1}(0) = Bun_{G,triv}(A_0)$ where the subscript ``triv'' means we are considering $G$-bundles together with a trivialization on the boundary of the disc.  The embedding $\Xa$ of $p^{-1}(\Cs)$ contains $p^{-1}(0)$ and completes this fiber and in turn allows one to complete $Bun_{G,hol}(C_u)$.  The cutting and gluing mentioned in this construction posses no problems in the holomorphic setting but the appropriate algebraic analogue requires more care and will be addressed in a follow up paper.

\subsection{Summary}
In section 2, we recall the construction of the wonderful compactification of an adjoint group and use the theory of toric stacks to generalize the construction to produce wonderful compactifications for all semisimple groups \ref{thm:stackGad}; this gives a representation theoretic construction to results in \cite{Martens}.  In section 3, we give basic definitions regarding loop groups and their representations.  Additionally, we generalize the results in section 2 to the loop group setting to produce an embedding $\Xa$.  In section 4, we discuss in detail the positive energy representations of the loop groups $LT$ for a torus $T$ and describe the embeddings of $LT$ in terms of combinatorics of positive definite forms. In section 5, we discuss in more detail the connection between $LG$ and bundles on curves or more generally, torsors on curves.

For the convenience of the reader I have included an index of notation which contains the notation defined in one section but used in multiple other sections.\\

\noindent{\bf Acknowledgements} I would like to thank Constantin Teleman for advising this project.  I also would like to thank Michel Brion, Mathieu Huruguen, Carlos Simpson, Michael Thaddeus for useful discussions.  I would like to especially thank Johan Martens for helpful discussions and also for pointing out some inaccuracies in an earlier version of this paper.
\section{Compactifications of Reductive Groups}\label{s:DC&P}
This section reviews the construction of the wonderful compactification of a semisimple group of adjoint type, we recall the basic results regarding its structure and in section \ref{sec:BrionKumar} describe an extension to reductive groups.  In section \ref{sec:stacky} we use the theory of toric stacks to provide wonderful stacky compactifications for a general semi simple group.  This method extends to to give compactifications of reductive groups but a further choice of a subdivision of a Weyl chamber must be made.  

\subsection{Wonderful Compactification of an Adjoint Group}\label{s:Gad} This section largely follows chapter 6 of \cite{Brion}.

\subsubsection{Definitions}\label{ss:Gdefs} Let $G$ be a semisimple group.  It has associated subgroups: a maximal torus $T$, opposite Borels $B, B^-$, their unipotent radicals $U,U^-$.  The character lattice we denote as $\ch{T}$, the co-character lattice we denote as $V_T$ and if $\mu \in \ch{T}, \eta \in V_T$ then the integer $\mu \circ \eta$ we denote as $\l \mu ,\eta \r$, $\l \mu ,\eta \r$, or $\mu(\eta)$.  Let $r = rk(G)$ and let $\a_1, \dotsc, \a_r$ be the positive simple roots.  Let $\omega_1,\dotsc, \omega_r$ be the fundamental weights.
\nomenclature{$T$, $B$,$B^-$,$U$,$U^-$}{Subgroups of a semisimple group}%
\nomenclature{$\a_1, \dotsc, \a_r$}{Roots of a semi simple group}%

\nomenclature{$\l \mu ,\eta \r$}{Paring between characters and co-characters}
For dominant weight $\la$ let $V(\la)$ denote the highest weight representation (HWR) of highest weight (HW) $\la$.  
\nomenclature{HWR}{Highest weight representation}%
\nomenclature{HW}{highest weight}%
By $V(\la)_\mu$ we denote the weight space of $V(\la)$ with weight $\mu$. When no confusion is likely to arise write simply $V_\mu$.  We can decompose 
\[
End(V(\la) ) = \bigoplus_{\mu, \chi \in \ch{T}} V_\mu \otimes V_\chi^*.
\]
Let $\P End(V(\la) )_0 \subset \P End(V(\la) )$ be the open subset consisting of points whose projection to $V_\la \otimes V_\la^*$ is not zero.  Define for a {\it regular} dominant weight $\la$
\begin{align}\label{Gad}
\begin{split}
X&:=\ol{G\x G.[id]} \in \P End(V(\la))  = \P \big[ V(\la)\otimes V(\la)^*\big]\\
X_0&:= X \cap \P End(V(\la) )_0
\end{split}
\end{align}
\nomenclature{$X$}{Wonderful compactification of $G_{ad}$}%
\nomenclature{$X_0$}{ Open cell of $X$}
Let $Z(G)$ denote the center of $G$.  It is routine to see that Stab($[id]$) is $Z(G) \x Z(G) \cdot \Delta(G)$ and consequently $X$ contains $G_{ad}:= G/Z(G) = \frac{G\x G}{Z(G)\x Z(G) \Delta(G)}$ as a dense open subset.  $X$ is the {\it wonderful compactification} of $G_{ad}$ and $X_0$ is the {\it open cell} of $X$.  A maximal torus for $G_{ad}$ is $\Tad = T/Z(G)$.  By $\ol{\Tad}$ we mean the closure of $\Tad$ in $X$ and $\oTad:= X_0 \cap \ol{\Tad}$. 
\nomenclature{$Z(G)$}{Center of a group $G$}%
\nomenclature{$G_{ad}$}{ Adjoint group of $G$}
\nomenclature{$\Tad$}{ Maximal torus for $G_{ad}$}
\nomenclature{$\ol{\Tad}$}{ Closure of $\Tad$ in $X$}
\nomenclature{$\oTad$}{ Affine toric variety, open cell of $\ol{\Tad}$}

Lastly, let $H$ be a reductive group and $Y$ a normal $H$-variety 
\begin{definition}\label{def:wond}
$Y$ is {\it wonderful of rank $r$} if $Y$ is smooth, proper and has $r$ normal corssing divisors $D_1,\dotsc, D_r$ such that the $H$-orbit closures are give by intersections $\cap_{i \in I} D_i$ for any subset $I \subset \{1,\dotsc, r\}$.
\end{definition}

The notation of $X,X_0, \ol{\Tad}, \oTad$ does not reflect the dependence on $\la$; this is justified by theorem \ref{thm:G_ad}.

\subsubsection{Properties of the Wonderful Compactification}
The main result of this subsection is theorem
\begin{thm}\label{thm:G_ad}
Let $X  = \overline{G}_{ad}$ be as in \eqref{Gad}.  Then 
\begin{itemize}
\item[(a)] $X$ is independent of $\lambda$.
\item[(b)] $X$ is smooth.
\item[(c)] $X - X_0$ is of pure codimension $1$; it consists of divisors that freely generate the Picard group.
\item[(d)] The boundary $X - G_{ad}$ consists of $r$ normal crossing divisors $D_1, \dotsc, D_r$ and the closure of the $G \x G$-orbits are in bijective correspondence with subsets $I \subset [1,r]$ in such a way that to $I$ we associate $\cap_{i \in I} D_i$.
\item[(e)] Any $G$ equivariant $X' \to X$ determines and is determined by a fan supported in the negative Weyl chamber.
\end{itemize}
\end{thm}
The proof will follow after a few lemmas which will also serve as our basic tools for studying stacky compactifications for reductive groups and loop groups. 

Let $t^{-\a}$ represent the regular function on $T$ given by the character $-\a$.
\begin{lemma}\label{l:T}
$\ol{T}_0 \cong \ec \C[t^{-\alpha_1},\dotsc, t^{-\alpha_r}] \cong \A^r$
\end{lemma}
\begin{proof}
See the proof of \ref{l:preimageX0}.
\end{proof}

\begin{lemma}\label{l:X0}
The action morphism $U^- \x U \x \ol{T} \to X$ sending $(u_1,u_2,t) \mapsto u_1tu_2^{-1} \in X$ maps isomorphically onto $X_0$. 
\end{lemma}
\begin{proof}
See \ref{p:loopX0}.
\end{proof}
\begin{lemma}\label{l:cl.or.}
Let $G$ be a reductive group and let $V$ a HWR of $G$.  Then the orbit of the HW vector is the only closed orbit in $\P V$
\end{lemma}
\begin{proof}
This follows from the Borel fixed point theorem \cite[Pg.272]{Borel} for reductive groups.
\end{proof}
\begin{rmk}When we discuss loop groups we give a different proof (lemma \ref{l:cl.or.K-M}) that circumvents the need for a Borel fixed point theorem in the Kac-Moody setting.\end{rmk}

\begin{proof}[proof of theorem \ref{thm:G_ad}] To prove $(a),(b)$, we note that lemmas \ref{l:T},\ref{l:X0} show $\ol{T}$ and $X_0$ are independent of $\la$ and smooth.  Observe that
\begin{equation}\label{eq:X0covers}
G \x G.X_0 = X  
\end{equation}
This follows from lemma \ref{l:cl.or.}: $G\x G.v_\la \otimes v_\la^*$ is the unique closed orbit in $\P$ and it clearly intersects $\P_0$.  Now $\P - G\x G.\P_0$ is $G \x G$-stable and closed therefore must be empty and hence we have \eqref{eq:X0covers}.  Consequently we have (b).

For (a), consider compactifications $X_\la$ and $X_\mu$ associated to weights $\la,\mu$.  Let $X_\Delta$ be the closure of $\Delta(G)$ in $X_\la \x X_\mu$.  The projection $p_\la \colon X_\Delta \to X_\la$ is equivariant and lemmas \ref{l:T},\ref{l:X0} imply $X_{\Delta,0}:= p_\la^{-1}(X_{\la,0}) \cong U^-\cdot \oTad\cdot U$; that is, the restriction of $p_\la$ to $X_{\Delta,0}$ is an isomorphism and therefor induces a $G_{ad} \x G_{ad}$ equivariant isomorphism on 
\[
\bigcup_{g \in G\x G}\  g.X_{\Delta,0} = X_\Delta \to X_\la = \bigcup_{g \in G\x G}\ g.X_0.
\]  
This proves $(a)$.

Part $(c)$ uses the result that if $U \subset Y$ is a dense open affine of any scheme $Y$ then $Y - U$ is of pure codimesion 1.  This is proved in \cite[2.4]{Vakil}.  In the case at hand, $X_0$ is affine and $Pic(X_0) = 0$; this immediately shows $Pic(X)$ is generated by the irreducible components of $X - X_0$.  A relation among these generator is a principal divisor $(f)$ which is invertible on $X_0$; such a function is a constant $c$ and $f -c$ is zero on a dense open set hence there are no relations.  Part $(d)$ is proved by applying lemma \ref{l:X0} to reduce it to the case $\ol{T}_0$ for which it is obvious. Part $(e)$ is addressed in the next section. 
\end{proof}
\begin{rmk}\label{rmk:Gad}
There are alternative constuctions of the wonderful compactification.  Starting with the Lie algebra $\mathfrak{g} \oplus\mathfrak{g}$ consider the point $\Delta(\mfg) \in Gr_{r,2r}$ and consider the closure of $G \x G.\Delta(\mfg)$ in $Gr_{r,2r}$; this gives another construction of the wonderful compactification of $G$; see \cite[pg.19]{DC&P}.

Alternatively in \cite{Vinberg}, Vinberg constructs an affine scheme $S$ with the action of $G \x T$ and identifies an open set $S^0 \subset S$ such that the GIT quotient $S^0//T$ is also isomorphic to the wonderful compactification.  Thaddeus and Martens use this construction and generalize it to provide stacky compactifications of reductive groups; \cite[$\S$ 5,6]{Martens}.
\end{rmk}
\subsection{Extension to Reductive Groups}\label{sec:BrionKumar} 
Let $G$ be a connected reductive group.  In \cite[6.2]{Brion}, Brion and Kumar define a $G$-embedding $Y$ to be a normal $G \times G$ variety containing $G = (G \times G)/diag(G)$ as an open orbit.  They call $X$ {\it toroidal} if the quotient map $G \to G/Z(G) = G_{ad}$ extends to a map $Y \to \overline{G_{ad}} =: X$= the wonderful compactification of $G_{ad}$. In fact, toroidal has a more general definition in the theory of spherical varieties but we will not need this level of generality.

They prove the following proposition; see \cite[6.2.4]{Brion}.
\begin{prop}\label{prop:fan2embedding}
Any toroical $G$ embedding is determined by its associated toric variety. In particular, given reductive group $G$, a maximal torus $T$ and a fan of the form $\Sigma = W\cdot \Sigma_0$ where $W$ is the Weyl group and $\Sigma_0$ is a fan with support in the negative Weyl chamber, one can construct a $G$-embedding and an equivariant morphism to $X$.
\end{prop}

\begin{proof}
Let $Y$ be a toroidal $G$-embedding.  We have a map $p \colon Y \to \ol{G_{ad}}$ and the associated toric variety is $p^{-1}(\ol{T_{ad}})$.  Equivariance implies that
\[
p^{-1}(X_0) \cong U^- \x p^{-1}(\ol{T_{ad,0}}) \x U.
\]
We now describe how to go from a Weyl equivariant toric variety to a $G$-embedding.  The above observation will show these are inverse constructions.

Let $\Sigma_0$ be as in the proposition and for a cone $\sigma \subset \Sigma_0$ let $\ol{T_{\sigma,0}}$ be the corresponding toric variety.  The idea is that given $\sigma$ we can construct a toroidal $G$-variety $\oG{\sigma}$ with a map $\oG{\sigma} \xrightarrow{p_\sigma} X$ such that $p_\sigma^{-1}(X_0) = U^-\ol{T_{\sigma,0}}U$; further if $ \tau \subset \sigma \subset \Sigma_0$ then we have an open immersion $\oG{\tau} \to \oG{\sigma}$.  Using these open immersions we can glue together the various $\oG{\sigma}$ to get the desired toroidal variety $\oG{\Sigma_0}$.

It remains to actually construct the spaces $\oG{\sigma}$ and show the gluing maps satisfy the co-cycle condition.  This we do in  \ref{l:preimageX0} and \ref{l:cocycle}
\end{proof}

\begin{lemma}\label{l:preimageX0}
Let $\Sigma_0$ be as in proposition \ref{prop:fan2embedding} and let $\sigma \subset \Sigma_0$ be a cone.  Then there is toroidal variety $\oG{\sigma}$ with a map $\oG{\sigma} \xrightarrow{p_\sigma} X$ such that
\[
p_\sigma^{-1}(X_0) \cong U^- \times \overline{T_{\sigma,0}} \times U
\]
and the closure $\ol{T_\sigma}$ of the torus in $T \subset \oG{\sigma}$ is given by the fan whose cones are $w\sigma$ for $w \in W$ and their faces.
\end{lemma}
\begin{proof}
We have $ \overline{T_{\sigma,0}} = \ec \C[\ch{T} \cap \sigma^\vee]$; choose generators $\mu_1, \dotsc, \mu_m$ of $\ch{T} \cap \sigma^\vee$.  Choose a regular dominant weight $\la$ such that $\la +\mu_i$ are all regular dominant.  Consider 
\[
\ol{G \x G.[id]} \subset \P\bigg( End(V(\la)) \bigoplus_i   End(V(\la+\mu_i) )\bigg) =: \P[\la,\la+\mu_i]
\]
and the rational map  $p_\sigma\colon \P[\la,\la+\mu_i] \dashedrightarrow \P End(V(\la))$.  Set $\oG{\sigma}$ to be the subset of $\ol{G \x G.[id]} $ such that $p_\sigma(\oG{\sigma}) \subset X$; it is evident that this is a $G\x G$ equivariant map.

Then restriction of $p_\sigma \colon \oG{\sigma} \to X$ to $G$ is just the projection $G \to G_{ad}$.  As $U^\pm$ do not intersect the center of $G$ we have $p_\sigma(U^\pm) = U^\pm$.  From the equivariance of $p_\sigma$ and from $X_0 \cong U^-\x \oTad \x U$ we have 
\[
p_\sigma^{-1}(X_0) = U^-\x p_\sigma^{-1}(\oTad) \x U
\]
So we are reduced to showing $p_\sigma^{-1}(\oTad) = \ol{T_{\sigma,0}}$. Let $\mu_1, \dotsc, \mu_m$ be generators for $\sigma^\vee$.  The image of $T$ in $\P[\la, \la + \mu_i]$ can be depicted as:
\[
\left[\begin{array}{cccc}
  \la(t) &  &  &  \\ 
   & {\la(t) \over \alpha_1(t)} &  &  \\ 
   &  & {\la(t) \over \alpha_2(t)} &  \\ 
   &  &  & \ddots \\ 
\end{array}\right] \oplus
\left[\begin{array}{cccc}
  \la(t)\mu_1(t) &  &  &  \\ 
   & {\la(t)\mu_1(t) \over \alpha_1(t)} &  &  \\ 
   &  & {\la(t)\mu_1(t) \over \alpha_2(t)} &  \\ 
   &  &  & \ddots \\ 
\end{array}\right]\oplus \dotsb \oplus
\left[\begin{array}{cccc}
  \la(t)\mu_m(t) &  &  &  \\ 
   & {\la(t)\mu_m(t) \over \alpha_1(t)} &  &  \\ 
   &  & {\la(t)\mu_m(t) \over \alpha_2(t)} &  \\ 
   &  &  & \ddots \\ 
\end{array}\right]
\]
The variety $p_\sigma^{-1}(\oTad)$ consists of those points where $\la(t) \ne 0$.  This is:
\[
\left[\begin{array}{cccc}
  1 &  &  &  \\ 
   & \alpha_1^{-1}(t) &  &  \\ 
   &  & \alpha^{-1}_2(t) &  \\ 
   &  &  & \ddots \\ 
\end{array}\right] \oplus
\left[\begin{array}{cccc}
  \mu_1(t) &  &  &  \\ 
   & {\mu_1(t) \over \alpha_1(t)} &  &  \\ 
   &  & {\mu_1(t) \over \alpha_2(t)} &  \\ 
   &  &  & \ddots \\ 
\end{array}\right]\oplus \dotsb \oplus
\left[\begin{array}{cccc}
  \mu_m(t) &  &  &  \\ 
   & {\mu_m(t) \over \alpha_1(t)} &  &  \\ 
   &  & {\mu_m(t) \over \alpha_2(t)} &  \\ 
   &  &  & \ddots \\ 
\end{array}\right]
\]
Let $\sigma_W$ denote the negative Weyl chamber in $V_T$.  Recall we require $\sigma \subset \sigma_W$ so all $-\alpha_i \in \sigma_W^\vee \subset \sigma^\vee$. It follows that all the diagonal entires are polynomials in $\mu_1,\dotsc, \mu_{m}$ so the projection to $\ec k[\mu_1(t),\dotsc,\mu_m(t)] \cong \ol{T_{\sigma,0}}$ is an isomorphism.  The last statement of the proposition follows because $\ol{T_\sigma} = p_\sigma^{-1}(\ol{\Tad})$ and 
\[
\ol{\Tad} = \bigcup_{w \in W} w\  \oTad\ w^{-1}.  
\]
This in turn follows from the $G \x G$ equivariance of the compactification $X$ which gives $W \x W$ equivariance of $\ol{\Tad}$.
\end{proof}

\begin{lemma}\label{l:cocycle}
If $\Sigma_0$ is as in \ref{prop:fan2embedding} and $\sigma_1,\sigma_2,\sigma_3 \subset \Sigma_0$ then the $\oG{\sigma_i}$ can be glued along $\oG{\sigma_i\cap\sigma_j}$.
\end{lemma}
\begin{proof}
Using lemma \ref{l:preimageX0} we reduce the gluing construction to the torus:  for any cone $\sigma \in \Sigma_0$ let $F_\sigma$ denote a choice for a set of generators for $\sigma^\vee$.  If $\tau \subset \sigma$ then we can choose $F_\sigma \subset F_\tau$.  Let $\la_\tau$ be such that $\la_\tau + F_\tau$ are all regular dominant weights.  There is a projection $\P[\la_\tau, \tau] \to \P[\la_\tau,\sigma]$ which maps $\oG{\tau} \to \oG{\sigma}$ and using lemma \ref{l:preimageX0}, equation \eqref{eq:X0covers} and the fact that $\ol{T_{\tau,0}} \to \ol{T_{\sigma,0}}$ is an open immersion we conclude the same for $\oG{\tau} \to \oG{\sigma}$.  

The cocycle condition follows immediately because it can be reduced to checking it for the tours for which it is manifestly obvious; namely, by $G\x G$ equivariance we can move any point $q \in \oG{\sigma_1\cap\sigma_2\cap\sigma_3}$ to lie in $\ol{T}_{\sigma_1\cap\sigma_2\cap\sigma_3,0}$.
\end{proof}

\subsection{Stacky Compactifications} \label{sec:stacky}
In general the fans $\Sigma$ that arise for reductive groups produce singular toric varieties and hence singular $G$-embeddings.  However these toric varieties are always smooth as stacks.  Here we describe a modification of the above construction that produces a smooth stack.  We briefly recall the basic theory of toric stacks and then incorporate them into the above construction.
\subsubsection{Preliminaries on Toric Stacks}
Following \cite{Anton}, we define {\it toric stacks} as $[Y(\Sigma)/Z]$ where $Y(\Sigma)$ is a normal toric variety with associated fan $\Sigma$ and $Z$ is a subgroup of the torus $T_\Sigma \subset Y(\Sigma)$.  The stack $[Y(\Sigma)/Z]$ contains the torus $T = T_\Sigma/Z$ as a dense open subscheme.  Just as in the theory of toric varieties, toric stacks are encoded by combinatorial data called stacky fans.  A {\it stacky fan} is a pair $(\Sigma, \beta \colon L \to N)$ where $L,N$ are lattices, $\Sigma$ is a fan in $L\otimes \R$, and $\beta$ is a homorphism of finite index.
\nomenclature{$(\Sigma, \beta \colon L \to N)$}{Stacky fan.}%
The equivalence between toric stacks and stacky fans is given as follows.  Given $[Y(\Sigma)/Z]$ we get a surjection $T_\Sigma \to T$ which induces a map $\beta \colon V_{T_\Sigma} \to V_T$.  Thus we get the stacky fan
\[
(\Sigma, \beta \colon V_{T_\Sigma} \to V_T).
\]
Starting from $(\Sigma, \beta \colon N \to L)$ we note that the hypothesis on $\beta$ implies that the dual morphism
\[
\beta^* \colon \hom(L, \Z) = L^* \to N^* = \hom(N,\Z)
\]
\nomenclature{$\beta^*$}{Dual of the map $\beta$ in a stacky fan}%
is injective. Consider the tori $T_\Sigma := \hom(N^*,\Cs) $ and $T := \hom(L^*,\Cs)$.  Dualizing $\beta^*$ we get a surjection
\[
0 \to Z(\beta) \to T_\Sigma \xrightarrow{ \ \ \hom(\beta^*,\Cs) \ \ } T \to 0
\]
\nomenclature{$Z(\beta)$}{Subgroup of a torus associated to a stacky fan.}%
and thus we get the toric stack $[Y(\Sigma)/Z(\beta)]$.
\nomenclature{$[Y(\Sigma)/Z(\beta)]$}{Toric stack associated to a stacky fan.}%

Here is how toric stacks appear in our situation.  Recall the notation $G,T, \a_1, \dotsc, \a_{r}$ from section \ref{s:DC&P}.  The Weyl chamber:
\[
C = \{ v \in V_{T}\otimes \R := V_{T,\R} | \alpha_i(v) \ge 0 \}
\]
is a fundamental domain for the action of the Weyl group on $V_T\otimes \R$ and defines a rationally smooth fan.  Let $u_1, \dotsc, u_r$ be generators of the rays of the fan $C$ and $M$ be the monoid they generate.  We get a homomorphism of lattices
\nomenclature{$u_1, \dotsc, u_r$}{Ray generators for the cone given by the Weyl chamber.}%
\begin{equation}\label{beta}
\beta \colon \Z^r \xrightarrow{e_i \mapsto u_i} V_T 
\end{equation}
and let $C'$ be the standard cone generated by the coordinate rays in $\Z^r\otimes \R$ giving rise to the toric variety $\A^r$.  We get a smooth toric stack $[\A^r/Z(\beta)]$ associated to the stacky fan $(C', \beta)$ whose coarse moduli space is the toric variety associated to the Weyl chamber in $V_T \otimes \R$.  For general toric stacks the subgroup $Z(\beta)$ is arbitrary but in the present situation $Z(\beta)$ is always finite.  In what follows we use the isomorphism $[\A^r/Z(\beta)] \cong [(\A^r \times T)/(\Cs)^r]$ where $(\Cs)^r$ acts diagonally on the product.

\begin{rmk}\label{ monoid}
Notice $\beta(C') = M$ and if $L$ is any lattice with a map $L\xrightarrow{\beta'} V_T$ and $M' \subset L$ is any monoid such that $\beta'(M')\subset M$ then we can lift $\beta'$:
\[
\xymatrix{
\Z^r\ar[r]^\beta & V_T\\
L \ar[ur]^{\beta'}\ar@{-->}^l[u]&
}
\]
and $l(M') \subset C'$.  This is used in the proof of theorem \ref{thm:stackGad}(d).
\end{rmk}

We use proposition \ref{prop:fan2embedding} together with the theory of toric stacks to give stacky compactifications of an arbitrary semi simple group $G$ analogous to the wonderful compactification of $G_{ad}$.

The stacky compactifications will be quotients of an embedding of the group $H = G \times (\Cs)^r$ where $r = rk(G)$.  The embedding is determined, via proposition \ref{prop:fan2embedding}, by a fan $C_\Delta \subset V_{T,\R} \oplus V_{(\Cs)^r,\R}$ in negative Weyl chamber of $H$.  Using the map $\beta$ in \eqref{beta} we get homormorphism $T \x (\Cs)^r \xrightarrow{(id,\beta)} T \x T$.  This map induces a map 
\[
V_{T,\R} \oplus V_{(\Cs)^r,\R} \xrightarrow{(id,\beta)_*} V_{T,\R} \oplus V_{T,\R}
\]
and the desired cone $C_\Delta$ is the inverse image of $-C \oplus C \subset V_{T,\R} \oplus V_{T,\R}$.  
\nomenclature{$C_\Delta$}{Cone depending on the fundamental weights of a semisimple group}%
By \ref{prop:fan2embedding} we get an $H$-embedding $Y(C_\Delta)$ 
\nomenclature{$Y(C_\Delta)$}{Quasi-projective $G \x (\Cs)^r$-embedding}; $Y(C_\Delta)$ is a strictly quasi-projective variety.

Set $T_\beta := (\Cs)^r$ and consider a subgroup of $H$ via $T_\beta \xrightarrow{\Delta} (\Cs)^r \x (\Cs)^r\xrightarrow{(id,\beta)} T\x (\Cs)^r \subset H$ the embedding $Y(C_\Delta)$ carries an action of $H \x H$ so we get an action of $T_\beta \x T_\beta$.  Identify $T_\beta$ with $T_\beta \x T_\beta/\Delta(T_\beta)$.

\begin{definition}\label{d:stack}
The stacky wonderful compactification of $G$ is 
\[
\mathcal{X} = \ol{G} := [Y(C_\Delta)/T_\beta].
\]
\end{definition}
The construction of the embedding $Y(C_\Delta)$ and thus $\mathcal{X}$ depend on a choice of a regular dominant weight $\la$ but just as in the case of the wonderful compactification different choices of $\la$ give isomorphic objects.
\begin{thm}\label{thm:stackGad}
\begin{itemize} Let $G$ be a semi simple group and $\mathcal{X}$ as in definition \ref{d:stack}.  Then:
\item[(a)] $\mathcal{X}$ is smooth and proper.
\item[(b)] $\mathcal{X} - \mathcal{X}_0$ is of pure codimension $1$ and we have an exact sequence
\[
0\to \Z^r \to Pic(\mathcal{X}) \to \hom(Z(\beta),\Cs) \to 0
\]
where the subgroup $\Z^r$ is generated by the irreducible components of $\mathcal{X} - \mathcal{X}_0$.
\item[(c)] The boundary $\mathcal{X} - G$ consists of $r$ divisors $D_1,\dotsc,D_{r}$ with simple normal crossings and the closure of the $G \x G$-orbits are in bijective correspondence with subsets $I \subset [1,r]$ in such a way that to $I$ we associate $\cap_{i \in I} D_i$.
\item[(d)] Let $u_1,\dotsc,u_{r}$ be generators of the rays of the Weyl chamber and $M$ be the monoid they generate.  Any $G$ equivariant $\mathcal{X}' \to \mathcal{X}$ determines and is determined by a fan supported in the negative Weyl chamber whose lattice points lie in $M$.
\end{itemize}
\end{thm}
\begin{proof}
Let $a_i = \beta^*(\a_i), w_i = \beta^*(\omega_i) \in \ch{V_\beta}$.  One can verify that $C_\Delta^\vee$ is generated by $(0,a_i),(\pm \omega_i, \pm w_i)$. It follows that the toric variety associated to $C_\Delta$ is isomorphic $T \x \A^r$ and using lemma \ref{l:preimageX0} we have
\begin{align*}
\mathcal{X}_0:=&[Y(C_\Delta)_0/T_\beta] = [(U^- \times T \times \A^r \times U)/T_\beta]\\
 \cong& U^- \times [(T \times \A^r)/T_\beta] \times U\\
\cong & U^- \x [\A^r/Z(\beta)] \x U.
\end{align*}
The open cell is smooth and its translates covers $\mathcal{X}$.  Further the restriction of $\mc{X}\to \ol{G_{ad}}$ is given by $U^- \x [\A^r/Z(\beta)] \x U \xrightarrow{(id,f,id)}U^- \x \oTad \x U$. The map $f$ is the composition $[\A^r/Z(\beta)]\to \ol{T_c}\xrightarrow{f'} \oTad$ where $\oTad$ is the course moduli space of $[\A^r/Z(\beta)]$.  We will show $f$ is affine and finite hence proper; this reduces to showing $f'$ is affine and finite; the first property is clear; finiteness follows because the roots of $G$ are finite index in the weights of $G$. We conclude that $\mathcal{X}$ is finite over the projective scheme $\ol{G}_{ad}$ hence (a).   

Let $X_c$ denote the coarse moduli space for $\mathcal{X}$; it comes with a surjection $X_c \xrightarrow{p_c} X$.  To prove (b) we use that $Pic(\mathcal{X}) = Cl(X_c)$; this is shown in remark 3.4 of \cite{Matt}.  Working in $X_c$ we see that as $X_{c,0} = U^- \x \ol{T_c} \x U$ is affine the complement is pure codimension $1$ and we have an exact sequence
\begin{equation}\label{ll}
\Z^r \xrightarrow{i} Cl(X_c) \to Cl(X_{c,0}) \cong Cl(\ol{T_c}).
\end{equation}
Notice that $p_c^* \colon \Z^r = Pic(X) \to Pic(X_c)$ is injective; this follows for the same reason given in \ref{thm:G_ad}(c).  Further after tensoring \eqref{ll} with $\Q$ we get a surjection $\Q^r \xrightarrow{i_\Q} Pic(X_c)_\Q$ and by the previous sentence $i_\Q$ must be an isomorphism and so $i$ is injective.

Now $Cl(\ol{T_c}) = Pic([\A^r/Z(\beta)])$ and the latter is equal to the group of $Z(\beta)$-equivariant line bundles on $\A^r$.  Any line bundle on $\A^r$ is trivial and an equivariant structure is uniquely determined by an element of $\hom(Z(\beta),\Cs)$.  Altogether this gives (b).

For (c) it follows similarly as for $\ol{G}_{ad}$ by reducing to the torus $\ol{T}_0$.  We can present $\ol{T}_0$ as $[\A^r/Z]$ where $Z$ is a finite group so the result follows by observing that it holds for the atlas $\A^r$.

To prove (d) we first note that any equivariant map determines a map $\mathcal{X}'_0 \to \mathcal{X}_0$ and consequently determines a map of torics stacks $[Y(\Sigma')/Z(\beta')] \to [\A^r/Z] = [Y(C'),\beta]$.  By \cite[3.4]{Anton} this corresponds to a diagram
\[
\xymatrix@R=3pt{
\Sigma' \ar[r] & C'\\ 
V_{T_{\Sigma'}}\ar[r]^l\ar[ddd]^{\beta'} & V_{T_\beta} = \Z^r\ar[ddd]^{\beta}\\
&\\
&\\
V_{T'} \ar[r]^f & V_T
}
\]
This gives one direction.  Conversely if we are given a fan whose lattice points lie in $M$ then by remark \ref{ monoid} we can lift this data to a diagram as above. In particular, we also have a map $T_{\Sigma'} \to T_\beta$ which we also denote by $l$

Now apply proposition \ref{prop:fan2embedding} for the group $H' = G \x T_{\Sigma'}$ and the cone $C'_\Delta$ which is the inverse image of $-C\oplus C$ under the map $V_T\oplus V_{T_{\Sigma'}} \xrightarrow{(id,\beta')}V_T\oplus V_T$.  This produces an embedding $Y(C'_\Delta)$ and we claim $[Y(C'_\Delta)/T_{\Sigma'}]$ is the desired $G$-equivariant embedding mapping to $\mc{X}$.

To see this note that $Y(C'_\Delta)$ has an $H' \x H'$ equivariant map to $Y(C_\Delta)$ ($H'$ acts on the latter via $H'\xrightarrow{id, l}H$). Further $Y(C'_\Delta)$ is a $T_{\Sigma'}$-bundle over $[Y(C'_\Delta)/T_{\Sigma'}]$ and therefore $Y(C'_\Delta)\x_{T_{\Sigma'}} T_\beta$ is a $T_\beta$-bundle and we have an equivariant map
\[
Y(C'_\Delta)\x_{T_{\Sigma'}} T_\beta \to Y(C_\Delta)\x_{T_\beta} T_\beta \cong Y(C_\Delta).
\]
By the definition of $\mc{X}$ we get a map $[Y(C'_\Delta)/T_{\Sigma'}] \to \mc{X}$.  
\end{proof}

\begin{rmk}
In general the generators of $M$ don't generate all the lattice points of the Weyl chamber $C$ thus there are fans supported in $-C$ which don't yield maps to $\mathcal{X}$; e.g. $-C$ itself. In particular, the canonical embedding $X_{can}$ of the group $G$ does not admit a map to $\mathcal{X}$ unless $\mathcal{X} = X_{can}$.
\end{rmk}

\begin{rmk}
As noted in the introduction, most of this result has been proved by Martens and Thaddeus in \cite{Martens}.  For parts (a),(d) see specifically \cite[4.2,6.4]{Martens}.  Part (c) seems implicit in the treatment given in \cite{Martens} but part (b) is new.  The construction given in \cite{Martens} uses a construction they call the Cox-Vinberg monoid $S$ it is an affine variety with an action of $G \x T$.  An open subset $S^0 \subset S$ is identified and their compactification is given by $[S^0/T]$.  Martens explained to me that these constructions are the same.  Specifically results of \cite{Vinberg} allow one to show $S_0$ is a toroidal embedding of $G \x T$ and thus comes from a fan supported in the negative Weyl chamber, this is the fan given by $C_\Delta$ above.  Alternatively one can also see the equivalence between the Vinberg construction and given construction using universal torsors; this is explained in \cite[3.2.4]{Brion2}.
\end{rmk}

\begin{rmk}
Of course this stacky construction extends to a general reductive group over $\C$.  But for a general reductive group the Weyl chamber is not a strongly convex rational polyhedral fan.  That is, there is not a toric variety associated to the Weyl chamber; this happens for $GL_n$.  Thus to get an embedding of such groups we must further subdivide the Weyl chamber. 

Kausz does this for $GL_N$ as follows.  Let $\phi \colon GL_N \to PGL_{N+1}$ be given by composing $g \mapsto \left(\begin{smallmatrix} g & 0 \\ 0 & (\det g)^{-1}  \end{smallmatrix}\right)$ with the natural map $SL_{N+1} \to PGL_{N+1}$.  This induces a map $GL_n \supset T_1$ to $T_2 \subset PGL_{n+1}$ and an induced map on co-character $\phi_*\colon V_{T_1} \to V_{T_2}$.  Let $\Sigma \subset V_{T_1}$ be the inverse image of $-C$ under $\phi_*$.  Then Kausz's compactification is the one associated to $\Sigma$ via proposition \ref{prop:fan2embedding}.  Certainly there are other possibilities.
\end{rmk}

\section{Extension to Loop groups}
The aim of this section is to prove the analogue of theorems \ref{thm:G_ad}, \ref{thm:stackGad} for loop groups.

As mentioned in the introduction there is both the algebraic loop group and the smooth loop group.  In this section we focus on the algebraic loop group.  To a large extend the statements we prove also hold for the smooth loop group.  We remark along the way what modifications, if any, are necessary to get a statement for the smooth loop group.  When we apply these results to studying the moduli stack of bundles on curves we will be more explicit about the distinction between $LG$ and $\sLG$.
\subsection{Preliminaries on Loop Groups}\label{ss:loopdef} 
Here we introduce the groups $LG$, $\Ga$. Let $G$ be a simple algebraic group over $\C$ with $\pi_1(G) = \pi_0(G) = 1$. The loop group $LG$ is the functor from $\C$-algebras to groups given by
\[
R \mapsto LG(R):= G(R((z)) ): = G(\ec R((z)) )
\]
The functor $LG$ is represented by an ind-scheme of infinite type; see \cite{Fa2}.  Closely related is $\Gsd(R) := \mathbb{G}_m(R) \ltimes LG(R)$ which, for $u \in \mathbb{G}_m(R), \ga \in LG(R)$ is defined by
\nomenclature{$\Gsd$}{Semi direct loop group}
$u\ga(z)u^{-1} = \ga(u z)$. Finally the group $\Ga$ is a central extension of $\Gsd$:
\[1\to \mathbb{G}_m \to \Ga \to \Gsd \to 1\]
\nomenclature{$\Ga$}{Affine Kac-Moody group.}
As we are working over $\C$ we'll write $\Cs$ for $\mathbb{G}_m$. Topologically, $\Ga = \Cs \x LG \x \Cs_c$ where the subscript $c$ indicates the factor is central.

To put this in perspective, semi simple groups are finite dimensional Kac-Moody groups. $\Ga$ is an infinite dimensional Kac-Moody group. Despite being infinite dimensional, $\Ga$ and $LG$ behave much like a semi simple group. In particular, $\Ga$ has associated root data.

The maximal torus  of $\Ga$ is $\Ta := \Cs \x T \x \Cs_c$ and characters are denoted as $(n, \mu ,l) \in \Z\oplus \ch{T} \oplus \Z$; sometimes we abbreviate $\Tsd:=\Cs \x T \subset \Gsd$.	
\nomenclature{$\Ta$}{maximal torus of $\Ga$}
\nomenclature{$\Tsd$}{Maximal torus of $\Gsd$}
If $\alpha_1, \dotsc, \alpha_r$ be the simple roots of $G$ then the {\it simple affine roots} of $\Ga$ are $(0,\alpha_1,0), \dotsc, (0,\alpha_r,0), \a_0 = (1,-\theta,0)$ where $\theta$ is the longest root of $G$.  
\nomenclature{$\a_0$}{Affine root of $\Ga$}
By abuse of notation we denote $(0, \alpha_i,0)$ simply by $\a_i$ so the simple roots of $\Ga$ are $\a_0,\dotsc, \a_r$.  

In addition we have a standard Borel $\Ba := \Cs \x \cB\x\Cs_c$ where $\cB$ along with related subgroups is defined below.
\begin{table}[htbp]
  \centering
  \begin{tabular}{@{} ll @{}}
$L^+G(R) := G(R[[z]])$  & $ L^{-}G(R) := G(R[z^{-1}])$\\
$\cB(R) := \{\ga \in L^+G(R)| \ga \in B(R) \mod z\}$ & $\cB^{-}(R) := \{\ga \in L^{-1}G(R)| \ga \in B^-(R) \mod z^{-1}\}$\\
$\cU(R) := \{\ga \in L^+G(R)| \ga \in U(R) \mod z\}$ & $\cU^{-}(R) := \{\ga \in L^{-}G(R)| \ga \in U^-(R) \mod z^{-1}\}$\\
  \end{tabular}
\end{table}
\nomenclature{$L^+G$,$L^-G$,$\cB$, $\cB^-$, $\cU$, $\cU^-$}{Subgroups of the loop group}
Further, the affine Weyl group $\Wa$ is by definition $N(\Ta)/\Ta$.  It is isomorphic to $W \ltimes V_T$ and we have the Birkhoff factorization \cite[pg.142]{Kumar} for $\Ga$
\nomenclature{$\Wa$}{Affine Weyl group}
\begin{align}\label{Birk}
\begin{split}
\Ga =& \bigsqcup_{w\in W^{aff}} \cU^-w\Ba
\end{split}
\end{align}
Of course this restrict to give a similar decomposition for $LG$ and $\Gsd$. Also we have an analogous group $\Gap$ where $LG$ is replaced by $L_{poly}G := G(\C[z^\pm])$.
\nomenclature{$\LpG$}{Polynomial loop group}
\nomenclature{$\Gap$}{Polynomial affine Kac-Moody group associated to $G$}

Similarly there is a group $\Gas$ where $LG$ is replaced by $\sLG = C^\infty(S^1,G(\C))$. If we identify $S^1$ with $\{|z| = 1\} \subset \C$ then, for example, $L^{+,sm}G = $ boundary values of holomorphic maps $\{ |z| < 1\} \to G$ and $L^-G$ is the subgroup of boundary values of holomorphic maps $\{|z|> 1\} \to G$, etc.  A reference for the smooth version of \eqref{Birk} is \cite[8.7.3a]{Segal}. 

In addition to the simple affine roots, we also have roots $(k,\alpha,0)$ where $\alpha$ is a root of $G$.  For $k \ne 0$ $\alpha = 0$ is allowed so in general the weight spaces are not 1 dimensional as $\dim \mfg_{(k \ne 0, 0,0)} = \dim \mft$.  $\Wa$ acts on the roots and one can talk about a Weyl chamber however it is standard practice to introduce instead Weyl alcoves. Namely, the roots are linear forms on $\R \oplus \mathfrak{t}_\R$ and we can identify them with affine linear forms on $\mathfrak{t}_\R$ by identifying the Lie algebra with $1 \oplus \mathfrak{t}_\R$.

For $\alpha \ne 0$ we can define affine hyperplane in $\mathfrak{t}_\R$ via
\[
H_{k,\a} = \{\zeta \in \mathfrak{t}_\R| \a(\zeta) = -k\}
\]
The complement of all the $H_{k,\a}$ is known as the {\it Weyl alcove decomposition} of $\mft_\R$.  $\Wa$ acts freely on $\mft_\R$ and permutes the alcoves in the decomposition.    A fundamental domain is given by the {\it positive Weyl alcove}
\[
Al_0 := \{\zeta \in \mathfrak{t}_\R| \a_i(\zeta) \ge 0,\ i = 0, \dotsc,r \}
\]
\nomenclature{$Al_0$}{Positive Weyl alcove}
\begin{rmk}\label{thickthin}
There are many parallels between $G$ and $\Ga$ but there are also clear differences. For example, the negative versions $L^-G,\cB^-,$ etc. are neither conjugate nor abstractly isomorphic to the positive versions.  This asymmetry is also reflected in the flag varieties of $LG = G((z))$: $LG/\cB$ is a projective ind-variety (union of finite dimensional projective varieties) and $LG/\cB^-$ is an infinite dimensional scheme.  The former we call a {\it thin} flag variety and the later a {\it thick} flag variety.
\end{rmk}

If $\omega_1,\dotsc, \omega_r$ are the fundamental weights of $G$ then the fundamental weights of $\Ga$ are $\omega_0 = (0,0,1),\ (0,\omega_1,1) \dotsc, (0,\omega_r,1)$. The definitions of dominant and regular weight carry over from the finite dimensional case.

\begin{prop}\label{rep1}
Let $(0, \lambda, h)$ be a regular dominant weight.  There exists a representation $V = V(0,\lambda,h)$ of $\Ga$ with the following properties
\begin{itemize}
\item[(a)] If $\mu$ is any other wight of $L$ then $\lambda - \mu$ is a sum of positive roots.
\item[(b)] $\lambda - \alpha_i$ is a weight of $V$ for all $i$.
\item[(c)] The stabilizer of the weight space $V_{\lambda}$ in $\P V$ is $\cB^{aff}$.
\item[(d)] The morphism $\Ga \bigl / \cB^{aff} = LG\bigl / \cB \to \P(V)$ given by $\gamma \cB \mapsto \gamma V_{\lambda}$ is injective and gives $LG\bigl / \cB$ the structure of a projective ind scheme; in particular $LG\bigl / \cB$ is closed in $\P(V)$.
\item[(e)] The action of $\Ga$ on $\P(V)$ factors through a faithful action of $\Ga \big/ Z(\Ga) = \Gsd/Z(G)$.
\end{itemize}
\end{prop}

\begin{proof}\noindent The existence claim is contained in \cite[13.2.8]{Kumar}.
\begin{enumerate}  
\item[(a)] \cite[1.3.22]{Kumar}.
\item[(b)] This follows form \cite[1.3.22]{Kumar} and the representation theory of $SL_2$.  Namely, for each simple root, consider the reflection $s_i(\lambda)$.  It is of the form $ \lambda - n \alpha_i$ for $n \ge 1$ and all the weights $\lambda - m \alpha_i$ for $0\le m < n$ are weights of the representation.
\item[(c)] \cite[7.1.2]{Kumar}

\item[(d)] \cite[Ch.7.1]{Kumar}

\item[(e)] \cite[13.2.8]{Kumar}.  More specifically we have the following commutative diagram
\[\xymatrix{ \Ga \ar[r]\ar[d]^{\psi} & GL(V)\ar[d]\\
\Gsd/Z(G)\ar[r] & PGL(V)}
\]
where $\psi$ is a surjective group homomorphism and $\ker \psi = Z(\Ga)$.
\end{enumerate}
\end{proof}

\begin{rmk}
The statements (a) - (c) and (f) for $\sLG$ are proved in \cite[ch.8,9,11]{Segal}.  The first half of (d) also holds \cite[8.7.6]{Segal} but there is no ind-statement to make about the flag varieties of the smooth loop group; they are infinite dimensional complex projective algebraic varieties.  
\end{rmk}

For more details on $\Ga$, its subgroup and its representation theory see \cite[chap. 13]{Kumar} or \cite{Bea} for $G = SL_r$. 

\begin{rmk}\label{PER}
There are three duals associated to a HWR $U = \oplus_\mu U_\mu$.  There is $U^*_{res} = \oplus_\mu U^*_\mu$ and $U^* = \prod_\mu U_\mu$ and a third called a positive energy representation (PER) $U_{pos}^*$.
\nomenclature{PER}{positive energy representation}%
PER result from completing $U_{res}^*$ with respect to a norm for which elements in $L^{sm}K$ act as unitary operators; $K$ is a compact form of $G$.  

Statements $(a),(b),(c),(e)$ still hold for $LG$ acting on $U^*$.  Again, statement $(d)$ is different because $LG/\cB^-$ is a thick flag variety. If we work with the polynomial loop group $L_{poly}G(R) = G(R[z^\pm])$ then proposition \ref{rep1} holds equally well for  both $U,U^*_{res}$.
\end{rmk}

We end this subsection with the following technical but important result. The result requires the notion of closure in $\P V$ when $\dim V$ is uncountably infinite.  We say $Y \subset \P U$ is {\it closed} if the inclusion $Y \to \P U$ satisfies the valuative criterion for completeness.  Colloquially $Y$ is closed if any limit point of $Y$ that exists in $\P U$ belongs to $Y$.  For example the map $\Cs \to U= \prod_{n \in \mathbb{N}} \C$ given by $t \mapsto (1, t,t^{-1}, t^2, t^{-2}, \dotsc)$ has no limit in $\P U$ as $t \to 0$ or $t \to \infty$ and consequently is a closed subset.

\begin{lemma}\label{l:cl.or.K-M}
Let $V$ a HWR of $\Ga$.  The orbit of a HW vector is the only closed orbit in $\P V$.  The orbit of a lowest weight vector is the only closed orbit in $\P V^*$.  The same statement holds for $\Gap$ acting on $\P V, \P V^*_{res}$.
\end{lemma}
\begin{proof}
Suppose the orbit of $v$ is closed in $\P V$.  Then for any $1$-parameter subgroup $\mathbb{G}_a \subset \Ga$ we have the closure $\overline{\mathbb{G}_a.v}$ is also in the orbit of $v$.  We use this observation to show that if $G.[v]$ is closed in $\P V$ then $\Ga.[v] = \Ga.[v_\la]$ where $\la$ is the HW of $V$.

For a positive simple root $\a_i$ set $U_i(b) := \exp(b\cdot X_i)$ with $b \in \C$. The representation $V$ is in particular an integrable representation of $\mfg^{aff}$ thus $X_i$ acts locally nilpotently on $V$.  So for any $v$ the vector $U_\a(b).v$ is a vector whose finitely many nonzero entries are a polynomial function of $b$.  It follows that the limit $b \mapsto 0$ exists in $\P U$ as well as the limit $b \mapsto \infty$; we write the latter as the limit $b \mapsto 0$ of $U_\a(1/b).v$.

Thinking of $V$ as a quotient of a Verma module we can write any vector as a sum over a finite index set $I$
\begin{equation}\label{eq:w}
v = \sum_{\vec n = (n_0, \dotsc, n_r) \in I} c_{\vec n}\  1_\la \otimes Y_0^{n_0} \dotsm Y_r^{n_r},
\end{equation}
where $1_\la$ spans a one dimensional representation of $\Tsd \x \Cs_c$ with weight $\la$. The action of $(1/b)X_0$ on $v$ lowers the exponent of $Y_0$, adds a factor of $(1/b)$ and leaves the other $Y_i^{n_i}$ unchanged.  It follows that
\[
v(0):= \lim_{b \to 0} U_{0}(1/b).v = \sum_{\vec n = (n_1, \dotsc, n_r) \in I'}c'_{\vec n}\  1_\la \otimes Y_1^{n_1} \dotsm Y_r^{n_r}
\]
Define $v(i)$ inductively by $v(i) = \lim_{b \to 0} U_{i}(1/b).v(i-1)$.  For any vector $v$ we have $[v(r)] = [v_\la] \in \P V$.  It follows that if $G.[v]$ is closed then $G.[v] = G.[v_\la]$. 

The same argument goes through for $U^*$ and $U^*_{res}$.  All that changes for $U^*$ is the sum \eqref{eq:w} is now potentially infinite but we can correct this by replacing $v$ with $\lim_{s \mapsto 0} \eta(s).v$ for a generic $\eta \in V_{\Tsd}$.  In this case $\eta$ acts nontrivially on the $v_\mu$ and $v' = \lim_{s \mapsto 0} \eta(s).v$ is supported in the weight spaces $V_\mu$ such that $\mu$ minimizes the function $\chi \mapsto \l \chi, \eta \r$.  This function is a quadratic function on a lattice so there are only finitely many $\mu$ that minimize it; see section \ref{sec:LT} for more details.  

We have shown if there is a closed orbit it must be that of a HW vector.  Let us show this is indeed closed.  Let $Y = \ol{\Ga.[v_\la]} \subset \P V$. Suppose $[w] \in Y - \Ga.[v_\la]$; we have $\ol{\Ga.[w]} \subset Y - \Ga.[v_\la]$ as $Y - \Ga.[v]$ is $\Ga$-stable and closed.  In particular, each $[w(i)] \in Y - \Ga.[v_\la]$ but this contradicts that $[w(r)] = [v_\la]$.
\end{proof}
\begin{rmk}
This result only relies on basic properties of representations of $LG$ also holds for $\sLG$ acting on $U^{pos}$; as $U^{pos}$ is a topological vector space the notion of closure is already well defined.
\end{rmk}

\subsection{The Wonderful Embedding of the Loop Group} The main result of this subsection is theorem \ref{thm:loopGad}.  We follow the same outline as in section \ref{s:Gad}.   

\subsubsection{Construction of $\Xa$}
Naively we would like to choose a regular dominant weight $\la$ of $\Ga$ and define
\[
\ol{\Gaa} := \ol{\Ga \x\Ga.[id]} \subset  \P (V(\la)\otimes V(\la)^*)
\]
however one immediately runs into the problem that $id\not\in V(\la)\otimes V(\la)^*$; instead we work with $V(\la)\widehat{\ox} V(\la)^*:= \prod_{\mu, \chi \in \ch{\Ta}} V_\mu \ox V_\chi^*$. We can write $\Ga = \bigcup_k (\Ga)_k$ where $(\Ga)_k$ is an infinite dimensional scheme always containing $Z(\Ga)$; the inversion map $g \mapsto g^{-1}$ is a closed embedding from $(\Ga)_k$ into some $(\Ga)_{k'}$ and in this way $(\Ga)_k^{-1}$ inherits a scheme structure.  We have a morphisms of schemes $(\Ga)_k \x (\Ga)_k^{-1} \xrightarrow{f_k} \P \big(V(\la)\widehat{\ox} V(\la)^* \big)$ given by $(g,h) \mapsto g.[id].h$.  Define

\begin{align}\label{Xak}
\begin{split}
\Xa_k &:= \ol{\im f_k}= \ol{(\Ga)_k.[id].(\Ga)_k^{-1}} \subset \P \big(V(\la)\widehat{\ox} V(\la)^* \big)\\
\Xa &:= \bigcup_k \Xa_k
\end{split}
\end{align}
\begin{prop}
$\Xa$ is an ind scheme of infinite type. $\Xa$ contains an open sub ind-scheme isomorphic to $\Gaa = \Gsd/Z(G)$.
\end{prop}
\begin{proof}
The first assertion is clear.  For the second, it suffices to show $\im f = (\Ga)_k/Z(\Ga)$.  In fact $f_k|_{(\Ga)_k \x id}$ is surjective and $(\Ga)_k \x id\cap f_k^{-1}([id]) = (\Ga)_k \x id \bigcap Z(\Ga) \x Z(\Ga) \Delta(\Ga) = (Z(\Ga),id)$; consequently $f_k((\Ga)_k \x id) = (\Ga)_k/Z(\Ga)$.  The equality $\Gaa = \Gsd/Z(G)$ follows from proposition \ref{rep1}(e): $Z(\Ga) = \Cs_c \x Z(G)$.
\end{proof}
\begin{rmk}
This construction immediately extends to $\Gap$ to give $\Xap$. For $\Gas$ the construction is even simpler.  Note $V(\la)_{pos},V(\la)^*_{pos}$ are Hilbert spaces and there is a good notion of tensor product for Hilbert spaces.  We still have $id \not \in V(\la)_{pos}\ox V(\la)^*_{pos}$ but for a generic 1 parameter subgroup $\eta \colon \Cs \to \Tsd$ and for $|t|$ sufficiently small we have $\eta(t).id \in V(\la)_{pos}\ox V(\la)^*_{pos}$.  The stabilizer of $[\eta(t) id]$ is conjugate to $Z(\Gas) \x Z(\Gas) \Delta(\Gas)$. Set $\Xas$ to be the closure of the orbit of $[\eta(t) id]$ in $\P(V(\la)_{pos}\ox V(\la)^*_{pos})$; it contains an orbit isomorphic to the adjoint of $\Gas$ as an open orbit. 
\end{rmk}
\nomenclature{$\Xap$}{Polynomial embedding}
\nomenclature{$\Xas$}{Smooth embedding}

\subsubsection{Properties of the Wonderful Embedding}
In this subsection we prove theorem \ref{thm:loopGad} which is the loop group analogue of theorem \ref{thm:G_ad}. Following section \ref{s:Gad}, we define $\Xa_{k,0}$ to be the open subset where projection to $V_\la \otimes V_\la^*$ is not zero.  $\Xa_0 = \cup_k \Xa_{k,0}$ is an open ind-subscheme of $\Xa$.
\nomenclature{$\Xa$}{Wonderful embedding of $\Gsd/Z(G)$}%
\nomenclature{$\Xa_0$}{ Open cell of $\Xa$}
$\Xa$ is the {\it wonderful embedding} of $\Gaa$ and $\Xa_0$ is the {\it open cell} of $\Xa$.  A maximal torus for $\Gaa$ is $\Tsad := \Tsd/Z(G)$.  By $\ol{\Tsad}$ we mean $\cup_k \ol{\Tsad}_k$ where the subscript denotes closure in  $\Xa_k$ and $\oTsad:= \Xa_0 \cap \ol{\Tsad}$; in fact $\oTsad$ is a constant ind-scheme. 
\nomenclature{$\Ga_{ad} = \Gsd/Z(G)$}{ Adjoint group of $\Ga$}
\nomenclature{$\Tsad$}{ Maximal torus for $\Ga_{ad}$}
\nomenclature{$\ol{\Tsad}$}{ Closure of $\Tsad$ in $\Xa$}
\nomenclature{$\oTsad$}{ Affine toric variety, open cell of $\ol{\Tsad}$}  

Let $t^{-\a_i}$ be the regular function on $\Tsad$ given by the character $-\a_i$.
\begin{prop}\label{p:loopT}
$\oTsad \cong \C[t^{-\alpha_0}, \dotsc, t^{-\alpha_l}] \cong \A^{l+1}$.  In particular $\oTsad$ is smooth and its fan is given by the negative Weyl alcove $-Al_0$; the fan $ \ol{\Tsad}$ is given by the Weyl aclove decomposition of $\mft^\ltimes_\R =  Lie(\Tsad)_\R$.
\end{prop}
\nomenclature{$\mft^\ltimes_\R$}{$\R$ Lie algebra of $\Tsd$}
\begin{proof}
Using proposition \ref{rep1}(b) one can use the same proof as in lemma \ref{l:preimageX0}.  The second statement follows because 
\[
\ol{\Tsad} = \cup_{w \in \Wa} w \oTsad w^{-1}
\]
and $-Al_0$ is a fundamental domain for the action of $\Wa$ on $\mft^\ltimes_\R$.
\end{proof}

Let $V = V(\la)$ be a HWR and let $v = v_{\la} \in V$ and $v^* \in V^*$ the dual vector.  Define $\P_v:=\{ v \ne 0\} \subset \P V$ and $\P_{v^*} = \{v^* \ne 0\} \subset \P V^*$.  The proofs of lemma \ref{l:loopX0} and proposition \ref{p:loopX0} are adapted from \cite[6.1.7]{Brion}; the extension for the loop group essentially requires replacing the Bruhat decomposition of a reductive group with the Birkhoff factorization of $LG$.
\begin{lemma}\label{l:loopX0} \noindent
 $LG.v \cap \P_v = \cU^{-}.v$ and $LG.v^* \cap \P_{v^*} = \cU.v^*$.  In particular $X_0$ is $\cU \times \cU^-$ stable. 
\end{lemma}
\begin{proof}
Let $\pi \colon V \to \C \cdot v$ be the projection.  Let $\pi_v(g) = \pi(g.v)$.  We have $\pi_v(g) = 0$ if and only if $g.v \cap \P_v = \emptyset$.  By the Birkhoff decomposition any $g = u_-.w.b$ with $u_- \in \cU^-$, $w \in \Wa$ and $b \in \cB$.  It follows that $\pi_v(g) = 0$ if and only if $\pi_v(u_-.w) = 0$. If $v' = v_\mu + \sum_{\nu < \mu}v_\nu$ and $u \in \cU^-$ then $u.v_\mu$ still has nonzero projection to $v_\mu$ and all other weights are still less than $\mu$.  Consequently $\pi_v(u_-.w) \ne 0$ if and only if $\pi_v(w)\ne 0$ if and only if $w = 1$.  Therefore
\[
LG.v \cap \P_v = \cU^-\cB.v \cap \P_v = \cU^-.v \cong \cU^-.
\]
The second equality follows because the stabilizer of $[v]$ in $\P V$ is $\cB$.  The same argument works with $LG$ acting on $V^*$.
\end{proof}
\begin{rmk}
Of course we can also replace $LG$ with $\Gsd$ or $\Ga$ in lemma \ref{l:loopX0}.
\end{rmk}
\begin{prop}\label{p:loopX0}
There is an $\cU \times \cU^-$ equivariant isomorphism
\begin{align*}
\cU^- \times \cU \x \ol{\Tsd_{ad,0}}  \xrightarrow{a}& \Xa_0\\ (l,u,t) \mapsto&\ l\cdot t \cdot u
\end{align*}
\end{prop}
\begin{proof}
First note that the restriction to $\cU^- \times \cU \x \Tsad$ is just the multiplication map and this is known to be open by the Birkhoff decomposition; consequently the morphism is birational.  The next step is to construct a $\cU \times \cU^-$-equivariant map $\Xa_0 \to \cU \times \cU^-$.

Let $\phi \in \Xa_0$ then thinking of $\phi$ as an endomorphism defined away from $\ker \phi$ we see that $\phi(v)$ is defined and is in $\P_v$.  On the other hand $\phi(v)$ is in the closure of the orbit of the HW vector, but this orbit is closed by lemma \ref{l:cl.or.K-M} hence $\phi(v) \in \Gsd.v$.  So by the previous lemma $\phi(v) = l.v$ for a unique $l$. We get a map $\Xa_0 \to \cU^-$ via $\phi \mapsto l$. Similarly we can get a map $\Xa_0 \to \cU$.  Altogether we have a map 
\begin{equation*}
\Xa_0 \xrightarrow{b} \cU^- \times \cU
\end{equation*}
The composition $\cU^- \times \cU \x \oTsad \to \Xa_0 \to \cU^- \times \cU$ is given by $(l,t, u) \mapsto (l,u)$.  
 
To finish we show $\cU^- \times  \cU^- \x b^{-1}(1,1) \xrightarrow{a} \Xa_0$ is bijective and $b^{-1}(1,1)$ = $\oTsad$.  For injectivity note that as $b^{-1}(1,1)$ is a subset of $\Xa_0$ it suffices to show that if $a(l,t,u) = a(l',t',u') = x$ then $u = u', l= l'$.  This follows
 \begin{equation*}
 (l,u) = b \circ a(l,t,u) = b(x) = b \circ a(l',t',u') = (l',u').
 \end{equation*}
Now surjectivity.  Let $\phi \in \Xa_0$ and $(l,u) = b(\phi)$.  Then $t := (l^{-1},u^{-1}).\phi \in b^{-1}(1,1)$, hence $(l,t,u)$ does the job.  

It remains to show $b^{-1}(1,1) = \oTsad$.  Clearly we have $\supset$ as $b^{-1}$ is closed and contains $\Tsad$ and as $a$ is birartional it follows that they have the same dimension.  Now $\pi_0(LG) = \pi_1(G)$ = 1.  Further, the map $G \to G/Z(G) =: G_{ad}$ induces a map $LG \to LG_{ad}$; the image is the connected component of the identity, in particular it is irreducible.  It follows that $\Xa$ and $\Xa_0$ are irreducible hence so is $\Xa_0/\cU \times \cU^- \cong b^{-1}(1,1)$.  Thus it must equal $\oTsad$
\end{proof}

The following result is crucial for proving theorem \ref{thm:loopGad}(b); the argument given below was conveyed to me by Sharwan Kumar
\begin{prop}\label{p:picUm}
$Pic(\cU^-) = Pic(\Xa_0) = 0$
\end{prop}
\begin{proof}
$Pic(\Xa_0) = Pic(\cU^-)\x Pic(\oTsad \cU)$. The second factor is $0$ because $\oTsad \cU$ is an infinite dimensional affine space.  So it remains to prove $Pic(\cU^-)=0$. 

For any $w \in \Wa$ we have a Schubert variety $\cB w \cB/\cB \subset LG/\cB$; set $\cU^-_w = \cU^- \cap \cB w \cB/\cB $. In fact $\cU^-\subset LG/\cB$ and we get an ind-structure on $\cU^- = \bigcup_n \cU^-_n$ where $\cU^-_n = \bigcup_{l(w)\le n}\cU^-_w$. We show $Pic(\cU_w) = 0$ for all $w$.

Fix $w$ and abbreviate $Y = \cU_w^-$. For any $k\in \mathbb{N}$ we have a short exact sequence $\Z/k \to \cO_Y^* \xrightarrow{f \mapsto f^k} \cO_Y^*$; using that $H^1_{et}(Y,\cO_Y^*) \cong Pic(Y)$ and looking at the long exact sequence in \'etale cohomology we get
\[
\dotsb \to  H^1_{et}(Y,\Z/k)\to Pic(Y) \to Pic(Y) \to H^2_{et}(Y,\Z/k) \to  \dotsb
\]
By the proof of \cite[7.4.17]{Kumar}, $Y$ is contractible and because $H^*_{et}( - ,\Z/k) = H^*_{singular}( - ,\Z/k)$ it follows that the outer terms vanish and $Pic(Y) \xrightarrow{L \mapsto L^{\ox k}}Pic(Y)$ is an isomorphism for any $k$. We now show $Pic(Y)$ is finitely generated and together with the previous statement it will follow that $Pic(Y)=0$.

$Y$ is a normal variety with $\dim Y = l(w)$ so by \cite[2.1.1]{Fulton} $Pic(Y)$ embeds in the Chow group $Pic(Y) \subset A_{l(w) -1}(Y)$. So reduce to showing $A_{l(w) -1}(Y)$ is finitely generated. By \cite[1.8]{Fulton} there is a surjection $A_{l(w) -1}(\cB w \cB/\cB)\to A_{l(w) -1}(Y)$. By \cite[19.1.11b]{Fulton} $A_{l(w) -1}(Y) = H_{2(l(w) -1)}(\cB w \cB/\cB,\Z)$ and finally the Bruhat decomposition implies the latter group is finitely generated.
\end{proof}

We know present two technical results that are in the proof of the main theorem.  Proposition \ref{p:Gorb} is well known in the finite dimensional case and the same is likely true for proposition \ref{c:Borb} but I couldn't find a reference.
\begin{prop}\label{p:Gorb}
The boundary of $\Xa$ is $\Gsd \x \Gsd$ stable and so breaks up into a disjoint union of orbits.  Proposition \ref{p:Gorb} implies the standard idempotents $e_I = \sum_{i \in I} e_i \in \oTsad \cong \A^{r+1}$ for $I \subset \{0, \dotsc, r\}$ lie in distinct $\Gsd \x \Gsd$ orbits.

For $J \subset [0,r]$ let $P_J^\pm$ be the associated opposite parabolics with standard Levi decomposition $P^\pm_J = L_J \cdot U_J^\pm$. Set $S(J) = \{(g_1,g_2) \in P_J \x P^-_J : \ \  g_1|_{L_J} = g_2|_{L_J}\} = \Delta(L_J) \ltimes (U_J \x U^-_J)$. Set $T(J):=Z(L_J) \x Z(L_J)$.  Then for $e_J = \sum_{j \not \in J} e_j \in \oTsad$ we have $Stab(e_J) = T(J) \cdot S(J)$ and
\[
\Xa = \bigsqcup_{J \subset [0,r]} Orbit(e_J) =  \bigsqcup_{J \subset [0,r]}  \Gsd \x \Gsd \big / T(J) \cdot S(J)
\]
Equivalently $Orbit(e_J)$ is an $L_{J,ad}$-bundle over $\Gsd/P_J \x \Gsd/P_J^-$ and $\ol{Orbit(e_J)}$ is an $\ol{L_{J,ad}}$-bundle over $\Gsd/P_J \x \Gsd/P_J^-$ where $\ol{L_{J,ad}}$ is the wonderful compactification of $L_{J,ad}$.\\

\end{prop}

\begin{cor}\label{c:Borb}
We can further decompose $Orbit(e_J)$ under $\cB^- \times \cB
\subset \Gsd \times \Gsd$:
\[ Orbit(e_J) = \bigsqcup_{( w_1, w_2) \in \Wa / W_J \times W_J \backslash
   \Wa} \bigsqcup_{w_3 \in W_J} \cB^- w_1 .e_J.w_3 w_2 \cB \]
where $W_J$ is the Weyl group of $L_J$. 
\end{cor}
Note when $J = \emptyset$ we get $P_J^{\pm} = L_J = \Gsd$
and $W_J = \Wa$ and the disjoint union becomes the usual Birkhoff decomposition. When $J = \{ \alpha_1, \ldots, \alpha_r \}$ then the parabolics are the Borels
and $W_J = 1$ and the disjoint union becomes the stratification by Shubert
cells of $\Gsd/\cB \times \cB^- \backslash \Gsd$.

We give the proofs of these results at the end of this sub section.

\begin{thm}\label{thm:loopGad}
Let $\Xa  = \ol{\Ga/Z(\Ga)}$ be as in \eqref{Xak}.  Then 
\begin{itemize}
\item[(a)] $\Xa$ is independent of $\lambda$. 
\item[(b)] $\Xa - \Xa_0$ is of pure codimension $1$.  It is a union of $r+1$ divisors that are Cartier and freely generate the Picard group.
\item[(c)] The boundary $\Xa - \Gsd/Z(G)$ consists of $r+1$ normal crossing divisors $D_0,\dotsc,D_r$ and the closure of the $\Ga \x \Ga$-orbits are in bijective correspondence with subsets $I \subset [0,r]$ in such a way that to $I$ we associate $\cap_{i \in I} D_i$.
\item[(d)] Any $G$ equivariant $X' \to \Xa$ determine and is determined by a weyl equivariant morphism of toric varieties $\overline{T}' \to \ol{\Tsd_{ad,0}}$.
\end{itemize}
\end{thm}
\begin{proof}
The proof of (a) is the same as \ref{thm:G_ad}(a).  For (b) we use \cite[2.4]{Vakil} which states that the complement of a dense open subset in any scheme
is of pure codimension one.  For every $n$, $\Xa_{n,0}$ is a dense affine scheme in $\Xa_n$ hence $\Xa_n - \Xa_{n,0}$ gives a sequence of compatible divisors and shows that $\Xa - \Xa_0$ itself is pure codimension $1$.  

Alternatively, by corollary \ref{c:Borb}, any $p \in \Xa - \Gsd/Z(G)$ can be expressed as $u^- w_1 \eta(0)w_2 b$ where $\eta\in -Al_0$ is a one parameter subgroup such that $\eta(0) = e_J$ for some $J$.  Clearly $p \in \ol{\cU^- w_1 \eta(s) w_2 \cB} = \ol{\cU^- w_1w_2 \cB}$. Using 
\[
E_i := \ol{\cU^- s_i \cB} = \bigsqcup_{v \ge s_i} \cU^- v \cB \ \ \ (\mbox{ closure in } \Ga)
\]
where $\ge$ is the Bruhat order on $\Wa$, see \cite[1.3.15]{Kumar}. Note that $E_i$ are codimension $1$ in $\Gsd/Z(G)$ and for some $i$ we have $p\in \ol{\cU^- w_1w_2 \cB} \subset \ol{E_i}$ (closure in $\Xa$).  It follows that $\Xa - \Gsd/Z(G)$ is the union of the $\ol{E_i}$ which is pure codimension $1$.

To show the $\ol{E_i}$ are Cartier we use that the $E_i$ are Cartier.  In fact there is a maximal parabolic subgroup $\mathcal{P}_i = L_i U_i$ and a line bundle $\mc{O}_i(1)$ on $\Ga/\mathcal{P}_i$ such that $E_i$ is the inverse image of the vanishing of a section $\sigma_i$ on $\Ga/\mathcal{P}_i$.  Moreover choosing a HWR $V_i$ such that $\mathcal{P}_i$ stabilizes the class of a HW vector $v_i$ then there is a morphism
\[
\Ga \to \Ga/\mathcal{P}_i \to \P V_i \ \ \ \ \ \ g \mapsto g\mathcal{P}_i \mapsto [g.v_i]
\]
and $\cO_i(1)$ is the pull back of $\cO(1)$ on $\P V_i$ and further the section $\sigma_i$ can be described as $\sigma_i(g) = \l v_i, g.v_i \r$; that is, the projection of $g.v_i$ to the $v_i$ weight space. We show that this section extends to all of $\Xa$.  Consider $\Xa \subset \P \big(V(\la)\widehat{\ox} V(\la)^* \big)$ and for $w_1,w_2 \in \Wa$ let $\Xa_{w_1,w_2}$ be the open subset where projection to $V_{w_1.\la}\ox V_{w_2.\la}$ is nonzero.  Then \ref{c:Borb} shows $\Xa$ is covered by $\Xa_{w_1,w_2}$.  

We extend $\sigma_i$ on each $\Xa_{w_1,w_2}$.  Namely, for $p$ in the boundary of $\Xa_{w_1,w_2}$ use \ref{c:Borb} again to write $p = \lim_{s\to 0} u \nu_1 \eta(s) \nu_2 b = \lim_{s\to 0} p(s)$. By definition of $\Xa_{w_1,w_2}$ we have $\lim_{s\to 0} \frac{\l \nu_1.v_\la, p(s) \nu_2.v_\la \r}{\l w_1.v_\la, p(s) w_2.v_\la \r}$ exists for any $\nu_1,\nu_2 \in \Wa$; indeed thinking of $p = p(0)$ as an infinite matrix then the limit in question is the entry of $p(0)$ in ``row'' $\nu_1.v_\la$ and ``column'' $\nu_2.v_\la$. Similarly, replacing $v_\la$ with $v_i$, we conclude the limit also exists.  We set $\sigma_i(p) = \lim_{s\to 0} \frac{\l \nu_1.v_\la, p(s) \nu_2.v_\la \r}{\l w_1.v_\la, p(s) w_2.v_\la \r}$.  

It remains to show the extension doesn't depend on the choice $p(s)$.  The numerator  reduces to
\begin{align*}
\sigma_i(p(s)) &=   \l v_i, \ \ u \nu_1 \eta(s) \nu_2 b.v_i \r\\
 &= \l \nu_1^{-1} u^{-1}.v_i, \ \ \eta(s) \nu_2.v_i \r= \l \nu_1^{-1} v_i, \ \  \eta(s) \nu_2.v_i \r
\end{align*}
and likewise for the denominator; altogether we see $\sigma_i(p)$ is a matrix coefficient for $\eta(0)$ where the limit is taken in $\P \big(V(i)\widehat{\ox} V(i)^* \big)$.  Varying $\eta$ just scales $\sigma$ so it just remains to show $\sigma_i$ doesn't depend on the choice of $\nu_1,\nu_2$.  Let $Stab_\la(\eta(0))$ denote the stabilizer of $\eta(0)$ for the action in $\P \big(V(\la)\widehat{\ox} V(\la)^* \big)$ and $Stab_i(\eta(0))$ the stabilizer for the action in $\P \big(V(i)\widehat{\ox} V(i)^* \big)$.  The required statement follows because $Stab_\la(\eta(0)) \subset Stab_i(\eta(0))$.

Finally, thanks to proposition \ref{p:picUm}, the same argument in \ref{thm:G_ad}(b) shows these divisors freely generate the Picard group.

For $(c)$ we note that any $\Ga \x \Ga$-orbit intersects $\oTsad$ along a unique $\Tsad \x \Tsad$-orbit.  This follows from the description of the stabilizers given in proposition \ref{p:Gorb} and the fact that the standard parabolic subgroups in a Tits system are not conjugate \cite[30.1]{Humph}.  Therefore we can reduce the statement to the torus for which it is obvious.

The same argument for \ref{thm:G_ad}(d) works for $(d)$ with one minor adjustment.  In the finite dimensional case we utilized regular dominant weights with no further qualification.  For $LG$ we need all the HWR to be of the same level.  Representations $(0,\la,l)$ of level $l$ are characterized by $\la(\theta^\vee) \le l$; here $\theta^\vee$ is the co-root associated to the longest root $\theta$ of $G$.  It follows for any finite set $\la_1,...,\la_m$ of regular dominant weights there is a fixed $l$ such that $(0,\la_i,l)$ are all regular dominant.   
\end{proof}

\begin{rmk}\label{rmk:loopGad}
The proofs for (a),(c),(d) hold for $\Xas$ without further comment.  Statement (b) is true for $\Xas$ but the argument is different: it is just necessary to show $Pic(\Xas_0) = 0$.  Given \ref{p:loopX0} and the fact that for $\sLG$ we have an abstract isomorphism $\cU \cong \cU^-$, the previous assertion reduces to $Pic(\cU) = 0$.  Let $\cO$ be the sheaf of holomorphic functions and $\cO^*$ its group of units.  Then $Pic(\cU) = H^1(\cU,\cO^*)$ and vanishing for the latter group follows from a standard argument that asserts line bundles on $\cU$ are topological together with the result \cite[8.7.4 (ii)]{Segal} which asserts that $\cU$ is contractible.  To show line bundles are topological let $\cO_{cts}$ be the sheaf of continuous functions on $\cU$ and $\cO_{cts}^*$ its group of units.  Then we have a commutative diagram (suppressing the mention of $\cU$)
\[
\xymatrix{
H^1(\cO) \ar[r]\ar[d] & H^1(\cO^*)\ar[r]\ar[d] & H^2(\Z)\ar[r]\ar@{=}[d] & H^2(\cO) \ar[d]\\
H^1(\cO_{cts}) \ar[r] & H^1(\cO_{cts}^*)\ar[r] & H^2(\Z)\ar[r] & H^2(\cO_{cts})
}
\]
The vanishing of $H^i(\cO_{cts})$ for $i \ge 1$ follows because $\cU$ is paracompact and the vanishing of $H^i(\cO)$ for $i \ge 1$ is proved in \cite[1.1]{Lamp}.  Altogether we have $H^1(\cO^*) = H^1(\cO^*_{cts})$ as required.
\end{rmk}

\begin{proof}[proof of prop. \ref{p:Gorb}]
That $T(J)$ is in the stabilizer follows from the description of $\oTsad$ given in the proof proposition \ref{p:loopT} and the fact that $Z(L_J) = \cap_{k \not \in J} \ker \a_k$.  So let us focus on the group $S(J)$.

The group $\Gsd\x \Gsd$ and in particular the standard parabolic subgroups are generated by the root subgroups $U(\a) \cong \mathbb{G}_a$ which it contains. So to check $S(J)$ is in the stabilzer it suffices to check it for the root subgroups it contains.  These break up into two cases.  Roots subgroups of the form $(U(\a),1)$ or $(1,U(\a))$ and those of the form $\Delta(U(\a))$.  We treat the first case; the second case follows similarly

In the first case $ \a$ is not a root of $Lie(L_J)$ and we have that $\a = \a_i + \a'$ for some $i \not \in J$.  It suffices to check that $X_\a.e_J = 0$. Recall $e_J$ is an idempotent of $End(V( \la))$ and we can express $e_J = \sum_j e_j \otimes e_j^*$ where $j$ ranges over some subset of the weights of the representation.  Therefore to show $X_\a.e_J = 0$ it suffices to show $X_\a.e_j = 0\ \forall j$.  Assume that $X_\a.e_j \ne 0$ for some $j$.  The weight $j$ has the property that $\la - j \in \sum_{i \in J} n_i \a_i$ with $n_i \ge 0$.  But if $e_\mu:= X_\a.e_j$ is not zero then it is a weight vector of weight $\mu = \a + j$.  But then $\la - \mu$ fails to be a sum of positive roots, contradiction.


This shows that $T(J)S(J)$ is contained in the stabilzer, and for codimensional reasons there can't be a higher dimensional group that stabilizes $e_J$.  As $\Gsd$ is connected, there can't be other components.

Finally, according to our conventions, $Orbit(e_{J'})\subset \ol{Orbit(e_{J})}$ off $J \subset J'$.  In this case from general properties of flag varieties we have surjections $LG/P_{J'} \to LG/P_{J}$ with fiber $P_J/P_{J'}$.  So indeed we get a surjection $\ol{Orbit(e_{J})}\xrightarrow{\pi_J} LG/P_{J} \x LG/P^-_{J}$.  Further the fibers are compactifications of $L_{ad,J}$.  To conclude that is it the wonderful compactification one can note $e_J$ is the identity in $End(V(\la_J))$ for a regular dominant weight $\la_J$ or $L_{J,ad}$.  

Alternatively, Let $m = rk(L_{J,ad})$ and $\beta_1, \dotsc, \beta_m$ its simple roots. Then for $p = \pi_J(e_J)$ we have $\pi_J^{-1}(p)$ is $L_{ad,J}\x L_{ad,J}$-stable, contains $2^m$ orbits, one open orbit isomorphic to $L_{ad,J}$ and the $L_{ad,J}\x L_{ad,J}$ orbit closures are given intersections of $\pi_J^{-1}(p) \cap \ol{Orbit(J \cup \beta_i)}$.  It follows that  $\pi_J^{-1}(p)$ isomorphic to the wonderful compactification of $L_{J,ad}$.\end{proof}

\begin{proof}[proof of corollary \ref{c:Borb}]
The given expression is stable under the action of $\cB^- \x \cB$ and disjointness of the expression is implied by the fact that $\cB^- \x \cB \cap \Wa = 1$.

Any $p \in Orbit(e_J)$ can be expressed as $g_1.e_J.g_2$ for $g_i \in \Gsd$. Recall the Levi factorization $P_J = L_JU_J$. A variant of equation \eqref{Birk} shows we can write $g_1 = v_1^-w_1l_1u_1$ and $g_2 = v_2l_2w_2u_2$ where $(v_i,w_i,l_iu_i) \in U_J^- \x \Wa \x L_JU_J$. For $l_3 =l_1l_2^{-1}$, proposition \ref{p:Gorb} implies $g_1.e_J.g_2 = v_1w_1 l_3.e_J.w_2u_2$.  

Using the Brikhoff decomposition for $L_J$ we write $l_3 = v_3 w_3 t^{-1} u^{-1}_3 \in U^-_L w_3B_L$ where $B_L = TU_L$ is a Borel for $L_J$ and $w_3$ is the Weyl group $W_L$ of $L_J$. In fact, by proposition \ref{p:Gorb}, after replacing $l_3$ with another element of $L_J$ we can assume $(w_1, w_2) \in \Wa/W_L \x W_L \backslash \Wa$. Altogether we are reduced to the expression
\[
g_1.e_J.g_2 = v_1w_1 v_3 w_3.e_Jt u_3.w_2u_2.
\]
The result will follow if we can show $(w_1v_3w_1^{-1}, w_2^{-1}u_3w_2) \in \cU^- \x \cU$.  For this it is enough to show that if $w\in \Wa$ and $-\a_i$ is a negative root then $w(-\a_i)$ is a positive root $\leftrightarrow$ there exists an reduced expression $w= s_{i_1}\dotsb s_{i}$.  The $\Leftarrow$ direction follows from \cite[pg.61]{Brion}.  

Using induction on the length of $w$ we can assume $w = s_j w'$ where $w'(-\a_i)$ is a negative root.  The hypothesis on $w$ imply $w'(-\a_i) = -\a_j$ and that $s_i$ must appear somewhere in an expression for $w$. We abuse notation below and use $w$ to indicate a reduced expression for $w$.

Consider the word $w'' = ws_i$. If $ws_i$ is not reduced then by the exchange property for reflection groups we can find a reduced expression for $w$ ending in $s_i$.  Therefore we can prove the result if we can rule out the possibility that $w''$ is reduced.  If it is reduced then $w''(-\a_i)$ is a positive root by the $\Leftarrow$ direction. But $w''(-\a_i) = w(\a_i) = - w(-\a_i)$ and $w(-\a_i)$ was assumed to be positive, a contradiction.
\end{proof}

\subsection{Completeness and Stacky Extension}
Let us now address the issue of completeness.  First we show we have a non constant regular function $\ol{pr_1}$ on $\Xa$.  Namely consider the character $\Gsd = \Cs \ltimes LG\xrightarrow{pr_1}\Cs$.  Let $V_{pr_1}$ be the associated $1$-dimensional representation.  For any $p\in \Xa$ we can find $\eta\colon \Cs \to \Tsd$ and $g_1,g_2 \in LG$ such that $p = p(0) =\lim_{s\to 0} g_1\eta(s)g_2$.  For any $v \in V_{pr_1} - 0$ define
\[
\ol{pr_1}(p) = \lim_{s\to 0} \frac{ g_1\eta(s)g_2.v}{v} = \lim_{s\to 0} \frac{\eta(s).v}{v} = \lim_{s\to 0} \l \eta(s), pr_1\r 
\]
Then $\ol{pr_1}$ extends $pr_1$ and we have a cartesian diagram
\begin{equation}\label{sfiber}
\xymatrix{
 \Gsd/Z(G)\ar[d]^{pr_1}\ar[r] & \Xa\ar[d]^{\ol{pr_1}}\\
 \Cs \ar[r] & \A^1
}
\end{equation}
The special fiber over $0 \in \A^1$ is the boundary and we denote it $\p \Xa$; there is a similar diagram for $\Xap,\Xas$.  This shows $\Xa$ is not compete but for $\p \Xap$ we have
\begin{thm}\label{thm:complete}
An arbitrary morphism
\[
\ec \C((s)) \to \p \Xap
\]
extends to $\C[[s]]$
\end{thm}

\begin{proof}
$\p \Xap$ is the union of $r+1$ components $\ol{Orbit(J = \{j\})}$ and any map from $\C((s))$ lands in one of them.  Proposition \ref{p:Gorb} for $\Xap$ gives a diagram
\[
\xymatrix{\ec \C((s))\ar[d]\ar[r]^{f} & \ol{Orbit(J)}\ar[d]^{\pi_J}\\
\ec \C[[s]] \ar[r]^{f' \ \ \ \ \ \ \ \ \ \ } & \LpG/P_J \x \LpG/P^-_J }
\]
The composition $\pi_J\circ f$ maps $\ec \C((s))$ into a finite dimensional projective variety and the valuative criterion for properness induces the map $f'$. In turn, properness of $\pi_J$ implies the result.
\end{proof}
\begin{rmk}
We do not have an analogous statement to \ref{thm:complete} for $\p\Xas$ because the smooth flag varieties are not complete.
\end{rmk}

\subsubsection{Stacky Extension}
As in the finite dimensional case we can construct a stacky extension.  With theorem \ref{thm:loopGad} in hand, one check easily that the constructions in sections \ref{sec:BrionKumar},\ref{sec:stacky} carry over to the loop group case.

There is only one slightly subtle point already mentioned in the proof of theorem \ref{thm:loopGad} is that when dealing with multiple dominant weights $\tilde \la, \tilde \mu$ they must be of the same level.  In this way, to a regular dominant weight $\tilde \la$ we can associate an ind stack $\mc{X}^{aff}$ containing $\Gsd$; there is analogously an open cell $\mc{X}_0^{aff}$ and 
\begin{thm}\label{thm:loopstackGad}
\begin{itemize}
\noindent \item[(a)] $\mathcal{X}^{aff}$ is formally smooth and independent of $\lambda$.
\item[(b)] $\mathcal{X}^{aff} - \mathcal{X}^{aff}_0$ is of pure codimension $1$ and we have an exact sequence
\[
0\to \Z^r \to Pic(\mathcal{X}) \to \hom(Z(\beta),\Cs) \to 0
\]
where the subgroup $\Z^r$ is generated by the irreducible components of $\mathcal{X} - \mathcal{X}_0$.
\item[(c)] The boundary $\mathcal{X}^{aff} - \Gsd$ consists of $r+1$ normal crossing divisors $D_0,\dotsc, D_r$ and the closure of the $\Gsd \x \Gsd$-orbits are in bijective correspondence with subsets $I \subset [0,r]$ in such a way that to $I$ we associate $\cap_{i \in I} D_i$.
\item[(d)] Let $u_0,\dotsc, u_r$ be generators of the rays of the Weyl alcove and $M$ be the monoid they generate.  Any $\Ga \x \Ga$ equivariant $\mathcal{X}' \to \mathcal{X}^{aff}$ determines and is determined by a fan supported in the negative Weyl chamber whose lattice points lie in $M$.
\item[(e)] The stabilizer of every point $p \in \mc{X}^{aff}$ is finite.
\end{itemize}
\end{thm}
\begin{proof}
Statements (a) - (d) follow readily from the corresponding statements for theorems \ref{thm:loopGad} and \ref{thm:stackGad}.  For statement (e) it suffices to look at points in $\mc{X}^{aff}_0$.  By proposition \ref{p:loopX0} (e) holds because it does for every point in the toric stack $\ol{\Tsd_0} \cong [\A^{r+1}/Z(\beta)]$.
\end{proof}
We can also replace $\Gsd$ with $\Cs \ltimes \LpG$.

\section{Some toric geometry of LT}\label{sec:LT}
An interesting aspect of the representation of loop groups is the appearance of a new intermediate abelian group.  More precisely, instead of asking for a maximal torus $T \subset LG$, we can ask for a maximal abelian subgroup $H \subset LG$.  The standard choice is $H = LT$ but there are others choices, see for example \cite[3.6]{loop}. 

Let $\Cs\ltimes \tilde LT$ denote the inverse image of $\LTsd = \Cs \ltimes LT$ under the projection $\Ga \to \Gsd$.  In fact, the group structure on $\Cs\ltimes \tilde LT$ depends only on the Killing form. More generally, given any positive definite form $Q$ on$Lie(T)$, one can construct an analogous central extensions of $\LTsd$. For any of these centrally extended groups, one can construct representations analogous to HWR or PER of $\Ga$ and therefore talk about a wonderful embedding of $\Cs\ltimes \tilde LT$.

In this section, we will study the closure of $\Tsd$ in these embeddings. This is the only interesting part in the sense that there is a group $N \subset \LTsd$ and $\LTsd \cong \Tsd \x N$ and the orbit of the identity under $N$ is closed.

The main result of this section is theorem \ref{thm:LT} which states that the closure of is an infinite type toric variety whose fan can be realized as the cone on a Voronoi diagram for $Q$; the latter is an object studied in combinatorics.

\subsection{The group $\tvt$}
Let $\Taf$ be a central extension of $\LTsd$ and let $V$ be an irreducible representation of $\Taf$; we get a ``wonderful embedding''
\[
\ol{\LTsd}:= \ol{\Taf \x \Taf.[id]} \subset \P (V\widehat{\ox}V^*).
\]
 It turns out understand the closure of $\Tsd$ in $\ol{\Taf \x \Taf.[id]}$ it suffices to look at the orbit under the subgroup $\Tsd \x V_T$ or more precisely its central extension $\tvt \subset \Taf$.

\nomenclature{$\Taf$}{central extension of $\Cs \ltimes LT$}%
\nomenclature{$\tvt$}{subgroup of $\Taf$}%
\nomenclature{$\tlt0$}{subgroup of $\Taf$}%
\nomenclature{$\ol{\LTsd}$}{Embedding of the $\LTsd$}

Let us begin by pinning down the group structure of $\tvt$. We can determine the structure on the central extension by looking at the case when $T$ sits inside a simple group $G$.  In this setting we have a linear map $V_T \to \ch{T}$, $\eta \mapsto \l \eta, \r$ given by the Killing form.  The affine Weyl group $\Wa = W \ltimes V_T$ has a natural action on the characters of $T^{aff}:= \Tsd \x \Cs_c = \Cs \x T \x \Cs_c$.  This action is computed in the proof of \cite[13.1.7]{Kumar}.  We identify characters with their derivatives so that if we have an element $t^{aff} = (e^r, e^t,e^c) \in T^{aff}$ then a character $\tilde \la = (n, \la, h)$ evaluates to 

\[
\la(t^{aff}) = e^{n r + \la(t) + h c} .
\]
With this notation the action of $\eta \in V_T \subset \Wa$ is
\begin{align*}
\tilde \la \xrightarrow{\eta}& \eta.\tilde \la =  \tilde \la + \Big( - \la(\eta) + {h \over 2}\l \eta,\eta\r \ ,\ \ h\l \eta, \r,\ \ 0 \Big) = \Big( n - \la(\eta) + {h \over 2}\l \eta,\eta\r, \ \ \ \la + h\l \eta, \r,  \ \ \ \ h  \Big)\\
\left[\begin{array}{c}
    e^r \\ 
    e^t \\ 
    1 \\ 
  \end{array}\right] \mapsto& \exp \Bigg(\big[ n - \la(\eta) + {h \over 2}\l \eta,\eta\r \big] \cdot r  + \big[\la(t) + h \l \eta, t\r \big] \Bigg)\\
  \mapsto& \exp \Bigg([n - \la(\eta) ]\cdot r \ \ + \ \ \la(t) \ \  + \ \  h\cdot [ \l \eta, t \r + {1 \over 2}\l \eta,\eta\r r] \Bigg)
\end{align*}  
We can use this to understand conjugation in the central extension, for the action of $\eta$ can also be interpreted as
\begin{align*}
T^{aff} \xrightarrow{\eta \circ ( - ) \circ \eta^{-1}}  T^{aff}  &\xrightarrow{\tilde \la} \Cs\\
\left[\begin{array}{c}
    e^r \\ 
    e^t \\ 
    1 \\ 
  \end{array}\right]     \mapsto
  \left[\begin{array}{c}
    e^r \\ 
    e^{t -\l \eta, \r r} \\ 
    c(\eta,t,r) \\ 
  \end{array}\right] &\mapsto \exp\bigg(n r + \la(t) -\la(\eta) r + h c(\eta,t,r)\bigg)\\
  & \mapsto \exp\bigg([n - \la(\eta)] \cdot r \ \  + \ \  \la(t) \ \  + \ \ h\cdot  c(\eta,t,r)\bigg)
\end{align*}

We are using the semi-direct product to compute the middle row above.  Comparing we see
\[
c(\eta,t,r) = \l \eta, t \r + {\l \eta,\eta\r\over 2} r 
\]
The cocycle $c(\eta,t,r)$ determines the group law on $\tvt$:  topologically $\tvt = \Cs \x T\times V_T \times \Cs_c$ and denoting elements by $(u,H,\eta,w)^T$, the group structure has the form:
\begin{align*}
\colv{u_1}{H_1}{\eta_1}{1} \cdot 
\colv{u_2}{H_2}{\eta_2}{1}
 =&  
 \colv{u_1}{H_1}{1}{1}\cdot
 \eta_1 \cdot
 \colv{u_2}{H_2}{\eta_2}{1}
\cdot  \eta_1^{-1} \eta_1\\
 = 
 \colv{u_1}{H_1}{1}{1}
\cdot 
\colv{u_2}{\eta_1^{-1}(u_2)H_2}{ \eta_1\eta_2}{c(\eta_1,H_2,\rho_2)}
=& 
\colv{u_1 u_2}{\eta_1^{-1}(\rho_2)H_1H_2 }{ \eta_1\eta_2}{c(\eta_1,H_2,\rho_2)}
\end{align*}
Thus we can get a group structure for every positive definite symmetric form $Q$ via
\[
c_Q(\eta,H,u) = Q( \eta, H) + u^{Q(\eta,\eta)\over 2}.
\]
To emphasize the dependence on $Q$, we let $\tilde Q$ denote $\tvt$ with its group structure determined by $Q$.
\nomenclature{$\tilde Q$}{The group $\tvt$ with group structure determined by $Q$}%
One can define $c_Q$ even if $Q$ is not positive definite.  However when we look at representations of these groups we want want the weight spaces of the non central $\Cs \subset \tilde Q$ to be bounded below and we also want the weight spaces to be finite dimensional; these conditions fail if $Q$ is not positive definite.

\subsubsection{Representations of $\tvt$}\label{sec:repLT}
In \cite{Segal}, Segal describes a representation of $\Cs \ltimes \tilde LT$ on a vector space of the form $V_1 \otimes V_2$ where $V_1$ is a countable dimensional representation of $\Cs \ltimes T \widetilde{\x} V_T$ and $V_2$ is an uncountable dimensional representation of $(LT)_0/T$.  This produces a positive energy representation of $\tilde LT$.  In this subsection we focus on $V_1$ because this is all that is necessary to understand the closure 
\[
\ol{\Tsd} \subset \ol{L^\ltimes T}
\]

Fix a positive definite symmetric form $Q$ on $V_T$; we view this as a linear map $V_T \to \ch{T}$.  Consider the group $\tilde Q$ from the previous section. We describe a representation of $\tilde Q$ on $V_1 = \oplus_{\chi \in Q(V_T)} V_{1,\chi}$.   In fact following the wonderful compactification recipe we look at the $\tilde Q \x \tilde Q$ orbit of the identity in $End(V_1)$.  

The embedding is of $\tvt/Z_Q$ for a finite subgroup $Z_Q \subset T$ defined below.  The embedding has 
\[
\tilde Q_{ad}:= \tilde Q \x \tilde Q \big / (Z_Q \Cs_c \x Z_Q \Cs_c)\Delta(\tilde Q) = \frac{\tvt}{Z_Q}
\]
\nomenclature{$\tilde Q_{ad}$}{adjoint group of $\tilde Q$}
as a dense open subset set.  The connected components are indexed by $V_T$ and the connected component of the identity is the closure of $\Tsd/Z_Q$. The group $Z_Q$ is $\ker \hom(Q,\Cs)$.  In more words, we have an isomorphism $T \cong \hom(\ch{T},\Cs)$ and therefore we get a map
\[
T \cong \hom(\ch{T},\Cs) \xrightarrow{\hom(Q,\Cs)} \hom(V_T,\Cs) \cong (\Cs)^{\dim T}
\]
and $Z_Q := \ker \hom(Q,\Cs)$.
\nomenclature{$Z_Q$}{Subgroup of torus assoicated to a quadratic form $Q$.}  
\begin{rmk}
If $T$ is a maximal torus of a semi simple lie group $G$ and $Q$ is the Killing form then $Z_Q = Z(G)$.  
\end{rmk}

Let us return to the representation $V_1$.  We can describe it as follows.  Let $v_\mu \in V_1$ be a weight vector with weight $\mu$, then:
\begin{align} \label{eq:LTrep}
\notag
\eta.v_\mu = v_{\mu + Q(\eta)}\\
t.v_\mu = \mu(t)v_\mu\\  \notag
u.v_\mu = u^{Q(\mu,\mu)/2}.v_\mu
\end{align}  
\begin{rmk}\label{rmk:minT}
One can check that this indeed defines a representation; a verification of this appears in \cite[pg.306]{Segal}.  An important property we can see from \eqref{eq:LTrep} is that if $\eta \in V_{\Tsd}$ is such that the composition $\Cs \xrightarrow{\eta} \Tsd = \Cs \x T \xrightarrow{pr_1} \Cs$ is given by a positive integer then
\[
lim_{s \to 0} \eta(s) \mbox{ exists}
\]
as an element of $\P \big(V_1 \otimes V_1^*\big)$.  To see this it is enough to check the function $\mu \to \l \mu, \eta \r$ from the weights of $V$ to $\Z$ has a well defined minimum value.  This follows because the last line of \eqref{eq:LTrep} shows the function $\mu \to \l \mu, \eta \r$ is quadratic with leading order coefficient positive.  Further the minimum is achieved at a finite number of weight spaces $\mu$.
\end{rmk}

We can identify the image of $\Tsd$ in $V_1\widehat{\otimes} V_1^*$ as diagonal matrices. Such matrices we can identify with $\hat V_1 = \prod_{\mu \in \ch{\Tsd} }(V_{1})_\mu$.  With this translation we have a representation of $\Tsd\to \hat V_1$ and the toric variety we are interested in is  $\ol{\Tsd_Q}:=\ol{\Tsd. \left[\prod_\mu 1\right]} \in \P \hat V_1$.
\nomenclature{$\ol{\Tsd_Q}$}{toric variety inside an embedding for the group $\tilde Q$}

\subsection{Voronoi and Delaunay Subdivisions and the fan}\label{sec:LTfan}
The goal of this subsection is to describe the fan for the toric variety $\ol{\Tsd_Q}$ defined in the previous subsection.  We begin by introducing the relevant combinatorics.

Let $W$ be a real vector space equipped with an inner product $Q$.  For $v,v' \in W$ the {\it distance} between $v,v'$ is
\[ 
|v-v'|_Q:= \bigg(Q(v-v',v-v')\bigg)^{1/2}.  
\]
Given a finite set of points $S \subset W$ one can form the Voronoi diagram or Voronoi partition associated to $S$.  This is partition $W = \cup_{s \in S}C_s$ where
\begin{equation*}
C_s = \{w \in W|\ \  |w- s|_Q \le |w-s'|_Q \forall s' \in S - s\}
\end{equation*}
We can replace $S$ with a regular lattice and apply the same procedure to get tiling of $W$.  For example if $W = \R^2$ and $Q = id$ and we consider the lattice $\Z^2$ then we get a tiling by squares centered around the points of $\Z^2$; the vertices of the squares are at the points $(1/2,1/2) + \Z^2$.  For general $Q$ these tilings are well studied and have many nice properties; see \cite{Voronoi}. 

Dual to Voronoi diagrams is the notion of Delaunay subdivision.  Let $L \subset W$ be a lattice.  For $r>0$ and $p \in W$ set $B(p,r) = \{w \in W :\ \ |w- p|_Q \le r \}$.  For any $p \in W$ set $r(p) = \min\{t \in \R_{\ge 0}| \ B(p,t) \cap L \ne \emptyset\}$ and 
\[
P(p) = \mbox{Convex Hull}( L \cap B(p,r(p)).
\]  
\nomenclature{$P(p)$}{convex hull of a finite number of points in a vector space.}
Then we get a subdivision of $W$ via $W = \cup_{p \in W}P(p)$.  Alternatively, the convex hull of $p_1, ..., p_n \in L$ is in the subdivision if and only if $C_{p_1} \cap \dotsb \cap C_{p_n} \ne \emptyset$

Let $C_s^0$ be defined by replacing $\le$ with $<$ in the definition of $C_s$.  Then $C_{\vec 0}^0 = \{p \in W| P(p) = 0\}$.  

Now let us explain the connection between this combinatorics and $\ol{\Tsd_Q} \subset \ol{ \tilde Q_{ad}}$.

\begin{thm}\label{thm:LT}
Let $\mathcal{F}$ be the fan of $\ol{\Tsd_Q}$.  Then $\mathcal{F}$ is contained in $V_{T,\R} \oplus \R_{>0}$. For $t>0$ identify $V_{T,\R}$ with $V_{T,\R}\oplus t$; there is a lattice $L_t \subset V_{T,\R}$ such that $\mathcal{F}$ is the cone on the Voronoi diagram associated to $(L_t, V_{T,\R}, Q)$. 
\end{thm}

\begin{proof}
Let $(a,\beta) \in V_{\Tsd} = \Z \oplus V_T$ be a co-character. Recall $\tilde Q$ acts on a vector space $V_1$, \eqref{eq:LTrep}.  Fix a coordinate $s$ on $\Cs$ then we have a map
\[
\Cs \xrightarrow{(a,\beta)} \Tsd \subset \P \hat V_1 \ \ \ \ s \mapsto \prod_{Q(\eta) \in Q(V_T)} s^{\fq{a}{\beta}{\eta}}
\]
The limit $s \to 0$ only exists if $a>0$.  Moreover, if $a>0$, then the limit exists for any value of $\beta$.  This gives the first claim.

Fix now $t \in \Z_{>0}$ and for every $\beta \in V_T$ consider the 1 parameter subgroup $diag(s^{\fq{t}{\beta}{\eta}})$.  Set
\[
p_\beta = \lim_{s\to 0} \prod_{\eta \in V_T} (s^{\fq{t}{\beta}{\eta}}) \in \P \hat V_1.
\]
For fixed $t$ and $\beta$, consider $\eta \mapsto \fq{t}{\beta}{\eta}$ as a function on $V_T$; it is quadratic and has a unique global minimum on $V_{T,\R}$.  Consequently there are only finitely many co-characters $\eta_1, \dotsc, \eta_{n} \in V_T$ where $\fq{t}{\beta}{\eta}$ attains its minimum, therefore 
\begin{align*}
p_\beta = \prod_{\eta \in V_T} (f(\eta)) \in \P \widehat{V_1} \  \ \ \  \  \ f(\eta) = \begin{cases}
1 \mbox{  if } \eta \in \{\eta_1,\dotsc, \eta_n\}\\
0 \mbox{ otherwise}
\end{cases}
\end{align*}
So in fact $p_\beta \in \P V_1$.  

Fix any $\beta \in V_T$ (with $t$ still fixed as above).  Then $\beta$ determines a set
\[
C(\beta) = \{ \beta' \in V_T | p_\beta = p_{\beta'}\}.
\]
such that the cone on $C(\beta)$ is a cone in $\mc{F}$.

We now relate this to the Voronoi subdivision $\cup_{\beta \in V_T} C_{\beta}$  associated to $(V_T \subset V_{T,\R},Q)$.  For this we need the auxiliary lattice
\[
L_t = \{\a \in V_{T,\R} | \a \mbox{ minimizes } \eta \mapsto  \fq{t}{\beta}{\eta}\}.
\]
Basic properties of quadratic functions imply that $L_t = \{\beta/t | \beta \in V_T\}$ so $L_t$ is indeed a lattice.  

We claim the $C(\beta)$ are exactly the non empty intersections $\cap_{\a_1, \dotsc, \a_m \in V_T} C_{\a_1} \cap \dotsb \cap C_{\a_m} \cap L_t$.  Indeed, assume $\beta/t,\beta'/t$ is in the intersection.  Then by properties above we have
\[
P(\beta/t) = P(\beta'/t)
\]  
but this means the lattice points $\a_1, \dotsc, \a_m$ are equidistant from $\beta/t$ and no other $\a \in V_T$ are closer.  This in turn implies that the function $\a \mapsto \fq{t}{\beta}{\a}$ is minimized on $V_T$ at exactly $\a_1, \dotsc, \a_m$. The same applies to $\beta'$.  In other words $C(\beta) = C(\beta')$.  One can run the statement in reverse to get the desired equality.
\end{proof}

\begin{rmk}
If we had worked with a the full representation $V_1\otimes V_2$ this would had the effect of adding a positive constant to the function $\a \mapsto \fq{t}{\beta}{\a}$ which doesn't effect where minimums occur.
\end{rmk}

\begin{rmk}
The fan of $\ol{\Tsd}$ in a HWR for $\tilde LG$ is not the cone on a Voronoi subdivision associated to the Killing form.  This is because a HWR of $\tilde LG$ is not irreducible as an $\tilde LT$ representation. In a regular representation of $\Ga$ all the irreducible representations of $\tilde L^\ltimes T$ appear as a direct summand. 

More generally, for a fixed $Q$, the HWR of $\tilde LT$ are indexed by the finite group $\ch{T}/Q(V_T)$.  We computed the toric variety of $\ol{\Tsd}$ in the representation corresponding to $0 \in \ch{T}/Q(V_T)$.  For nonzero elements in $\ch{T}/Q(V_T)$ the fan of $\ol{\Tsd}$ is also a cone on a Voronoi diagram but the lattice is different.\end{rmk}

\subsubsection{Example}\label{ex:LT} Consider the case $T = \Cs$ then the form $Q$ above is a positive integer; we take $Q = 2$.  This choice comes from the killing form on $\Cs \subset SL_2$.  Let us consider the Voronoi subdivision associated to $\Z$ in the normed vector space $(\R, Q)$.  Clearly the set of points closest to $n \in \Z$ consists of the line segment $[n-1/2,n+1/2]$; in the notation of \ref{sec:LTfan}, $C_n = [n-1/2,n+1/2]$.  The fan of $\Tsd \subset \ol{L^\ltimes \Cs/\pm 1}$ is shown in figure \ref{pic:ltfan} where we have identified $\R$ with $\R \oplus 1$.

\begin{figure*}[htm]
\centering
\includegraphics[scale=0.65]{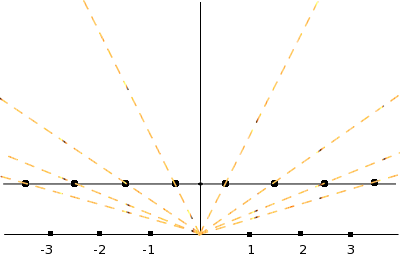}
\caption{Cone on the Voronoi subdivision.  The lattice $\Z$ is shown in square dots and the lattice points of the Voronoi subdivision is shown in circle dots.  The rays of the fan are shown in dashed dots.  Note only a finite number of rays are drawn but every point in the upper half plane lies on a ray or in between two rays.}
\label{pic:ltfan}
\end{figure*}
To differentiate $T = \Cs$ from the $\Cs$ in the semidirect product we write $T = GL_1$.  The associated toric variety $\ol{\Tsd}$ has a map to $\A^1$ extending the projection $\Cs \ltimes GL_1 \xrightarrow{pr_1} \Cs$ and fits into the following diagram
\[
\xymatrix{
\Cs \x GL_1\ar[d] \ \  \ \ar@{^(->}[r] & \ol{\Tsd}\ar[d] & \ar@{_(->}[l] \ \ \ \bigcup_{j \in \Z} \P^1_j \ar[d]\\
\Cs \  \ar@{^(->}[r] & \A^1 &   \ar@{_(->}[l] \ 0
}
\]

\section{Bundles on Curves}\label{sec:bundles}
In this section discuss the connection between $LG$ and the moduli stack $Bun_G(C)$ of principal $G$ bundles on a curve $C$.  We explain the connection between the embedding $\cXap$ and bundles on nodal curves as well as a connection between $\cXap$ and a construction of a completion of $Bun_G(C)$ on nodal curves due to Faltings.

\subsection{The double coset construction} 
Fix a smooth curve $C/\C$ and a point $p \in C$. Let $Bun_G(C),Bun_{G,hol}(C)$ denote the moduli stacks of algebraic and holomorphic principal $G$-bundles respectively. 

\subsubsection{Holomorphic Version}
Fix a local isomorphism $z$ from $\{|z| > 1/2 \}$ in $\infty \in \P^1(\C)$ to a neighborhood around $p$.  Set $D_p = \{|z| > 1\}$ and $C^* = C - \ol{D_p}$ so that $\ol{D_p} \cap \ol{C^*} = S^1$ identified with $\{|z| = 1\}$.   Consider now the smooth loop group $\sLG = C^\infty(S^1, G)$.  Define
\[
L^{sm}_C G = Hol(C^*, G) \subset \sLG,
\]
as the subgroup of loops that extend to give a holomorphic map $C^* \to G$.
We have the following theorem proved by Atiyah:
\begin{thm}
Let $G$ be a connected topological group.  The set of isomorphism classes of holomorphic principal $G$-bundles on $C$ is equal to 
\[
L^{sm}_CG\backslash \sLG/L^{sm,+}G
\]
\end{thm}
\begin{proof}
\cite[8.11.5]{loop}
\end{proof}
\begin{rmk}\label{hol}
One proves the result by noting a $G$-bundle on $C^*$ is holomorphically trivial.  Additionally the bundle is trivial on $D_p$ and so the bundle is determined by a ``transition function'' $\ga \in \sLG$; modding out by the two subgroups amount to accounting for changes of trivialization.  
\end{rmk}

\subsubsection{Algebraic Version}
In the algebraic setup the point a local coordinate $z$ at $p \in C$ is used to identify $\C((z))$ with $frac(\widehat{\cO_{C,p}})$ and thus identify $LG$ with $G(frac(\widehat{\cO_{C,p}}))$.  In this way $\hom_{alg}(C -p,G)$ embeds as a subgroup $L_CG$ of $LG$.  More precisely, $L_C G$, is the ind algebraic group which associates to a $\C$-algebra $R$ the group 
\[
L_CG(R) := G((C-p)_R)
\]
of algebraic maps $(C - p) \x_\C \ec R \to G$.

The algebraic double coset construction is
\begin{thm}\label{algBun} Let $G$ be a semi simple group. There is a canonical isomorphism of stacks
\[ 
Bun_G(C) \cong L_C G\backslash LG/L^+G
\]
\end{thm}
\begin{proof}
This was proved for $G = SL_n$ by Beauville and Lazlo in \cite{Bea} and for general $G$ by Lazlo and Sorger in \cite{Sorger}.
\end{proof}
The restriction to semi simple $G$ allows us to conclude $G$-bundles are trivial over $C-p$. To see this for $G = SL_r$ note that we are considering locally free sheaves $E$ with $\det E = \cO_C$.  Over $C-p$, $E$ is just a projective module over $A = \cO(C- p)$ which is a Dedekind domain.  A structure theorem for modules over a Dedekind domain gives that
\[
E \cong A^{\oplus rk(E) - 1} \oplus \det E = A^{\oplus rk(E)}.
\]

For general semi simple $G$ choose a faithful irreducible representation $G \subset SL(V)$ and identify a principal $G$-bundle with an $SL(V)$-bundle together with a reduction of the structure group to $G$.  Then over $C-p$ the $SL(V)$ bundle is trivial and the reduction to $G$ is preserved so the $G$ bundle is trivial on $C-p$.  More generally, for a family of curves, we have the following result due to Drinfeld and Simpson \cite{Drin}.  This is the main ingredient in the proof of theorem \ref{algBun}.
\begin{thm}
Let $S$ be a scheme and $C$ a smooth proper scheme over $S$ with connected geometric fibers of pure dimension $1$ and let $G$ be a semisimple group.  Let $D$ be a subscheme of $C$ such that the projection $D \to S$ is an isomorphism.  Set $U = C - D$.  Then for any $G$-bundle $F$ on $C$ its restriction to $U$ becomes trivial after a suitable faithfully flat base change $S' \to S$ with $S'$ being locally of finite presentation over $S$.  If $S$ is a scheme over $\Z[n^{-1}]$ where $n$ is the order of $\pi_1(G(\C))$ then $S'$ can be chosen to be \'{e}tale over $S$. 
\end{thm}

\subsection{The connection with $\Xap$ and bundles on nodal curves}
Here we focus on a modular interpretation of certain orbits in $\Xap$. This is an embedding of the group $\Gsd/Z(G)$ which is the connected component $(L G_{ad})_0$ of loops in $G_{ad}$.  Consequently the interpretation is in terms of $G_{ad}$-bundles on curves.

To begin this discussion consider the formal neighborhood of a node on a fixed nodal curve $C$.  The neighborhood is isomorphic to $D =  \ec \C[[x,y]]\big / xy$; it has a unique closed point $p = (x,y)$ and $D - p = D^* = \ec (\C((x)) \x \C((y)))$.  For any $\C$-algebra $R$ set $D_R = \ec R[[x,y]]\big / xy$ and define $D^*_R$ similarly.

Define a stack $Bun_{G,\tau}D$ on $\C$-algebras given by
\[
Bun^0_{G_{ad},\tau}D(R) =\bigg \l  G-\mbox{bundles} \begin{array}{c}
    P \\ 
    \downarrow\\ 
    D_R \\ 
  \end{array}, \tau\colon P|_{D^*_R} \widetilde{\to}  D_R^* \x G
\bigg \r
\]
 where angled brackets denote groupoid and the superscript $0$ denotes connected component of the trivial bundle. 	
 
Set $\zG = \Cs\ltimes \LpG$. The $\zG \x \zG$ orbits of $\Xap$ are parameterized by subsets $J\subset \{0,\dotsc, r\}$ (see proposition \ref{p:Gorb}).  In what follows we focus on $J$ with $|J| = 1$ and abbreviate $Orbit(\{ j\})$ by $\bO_j$; further define the subgroup $H_j \subset \zG \x \zG =: \zG^{\x 2}$ by $\bO_j = \frac{\zG^{\x 2}}{H_j}$.
\nomenclature{$\bO_j$}{Orbit in $\Xa$ associated to the subset $\{j\}$}
\nomenclature{$H_j$}{Subgroup associated to the orbit $\bO_j$}
Consider $\bO_j$ as a stack by sheafifying the presheaf $R \mapsto \frac{\zG^{\x 2}(R)}{H_j(R)}$.
 \begin{prop}
$Bun^0_{G_{ad},\tau}D$ is represented by $\mathbf{O}_0 \subset \Xap$.
\end{prop}
\begin{proof}
Any subgroup of $LG_{ad}$ is to be interpreted as the connected component of the identity; in particular, $LG_{ad}$ stands for $(LG_{ad})_0$; we don't introduce notation for this to cut down on the notation. 

Set $\hat N = \ker(G[[z]] \to G)$. Notice $G_{ad}(D_R) \xrightarrow{(x,y)\to (z,z)} G_{ad}(R)\ltimes (\hat N(R) \x \hat N(R))$ is an isomorphism. This is because $\ga \in G_{ad}(D_R)$ is equivalent to $(\ga_1,\ga_2) \in G_{ad}[[x]] \x G_{ad}[[y]]$ such that $\ga_1(0) = \ga_2(0)$.

Using faithfully flat descent one can construct inverse maps from $Bun^0_{G_{ad},\tau}D(R)$ to $\frac{G_{ad}(R((z))) \x G_{ad}(R((z)))}{G_{ad}(R)\ltimes (\hat N(R) \x \hat N(R))}$.  For example, starting from $(P,\tau)\in Bun^0_{G_{ad},\tau}D(R)$ we note for a faithfully flat extension $R\to R'$  we get $P_{R'} \cong D_{R'}\x G_{ad}$; comparing with $\tau_{R'}$ gives a $\ga_{R'}\in G_{ad}(R'((x))) \x G_{ad}(R'((y)))$. Set $R'' = R'\ox_R R'$; the two different pullbacks of $P_{R'}$ to $D_{R''}$ are isomorphic hence differ by an element in $G_{ad}(D_{R''}) \cong G_{ad}(R'')\ltimes (\hat N(R'') \x \hat N(R''))$ hence $\ga_{R'}$ descends to a well defined coset $[\ga_{R}] \in \frac{G_{ad}(R((z))) \x G_{ad}(R((z)))}{G_{ad}(R)\ltimes (\hat N(R) \x \hat N(R))}$.  The other direction is similar.

It remains to show $\frac{G_{ad}((z)) \x G_{ad}((z))}{G_{ad}\ltimes (\hat N \x \hat N)} \cong \bO_0$.  By proposition \ref{p:Gorb}, $H_0= Z(G) \x Z(G) G\ltimes (N \x N^-)$ where $N = \ker(G[z] \to G)$.  Consequently $\bO_0 = \frac{G_{ad}[z^\pm] \x G_{ad}[z^\pm]}{G_{ad}\ltimes (N \x N^-)}$.  From \cite[pg.231]{Kumar} we get $G_{ad}((z))/G_{ad}[[z]] \cong G_{ad}[z^\pm]/G_{ad}[z] \cong G_{ad}[z^\pm]/G_{ad}[z^{-1}]$ and the result follows.
\end{proof}

Fix a nodal curve $C$ with a unique node $x$ such that $C-x$ is affine.  Define $Bun^0_{G_{ad},C-p}C$ to be the stack parametrizing $G_{ad}$-bundles on $C$ with a trivialization on $C^* = C-x$. A simple variant of the above proof gives
 \begin{prop}\label{p:O0}
$Bun^0_{G_{ad},C^*}C$ is represented by $\mathbf{O}_0 \subset \Xap$.
\end{prop}

We can extend the above results to $\bO_j$ for $j \ne 0$ by introducing $\cG$-torsors.  

\subsection{$\mc{G}$-torsors}

In general, given a curve $C \to B$ and a sheaf of group $\cG$ on $C$ we define a {\it $\cG$-torsor} to be a sheaf of sets $\mathcal{F}$ on $C \to B$ together with a right action of $\cG$ such that (1) there is a finite flat cover $\{C_i \to C\}$ such that $\mathcal{F}(C_i) \ne \emptyset$ and (2) the action map $\cG \x \mathcal{F} \to \mathcal{F} \x \mathcal{F}$ is an isomorphism.
\nomenclature{$\cG$}{sheaf of groups on a family of curves}  

Given a principal bundle $F$ on $C$ we can consider the sheaf of groups $\cG^{std}$ defined by $U \mapsto \hom_{alg}(U,G)$ and the sheaf of sets $U \mapsto Sect(U,F|_U)$. In this way we associate to each principal $G$ bundle a $\cG^{std}$-torsor.  In fact there is a perfect dictionary between $\cG^{std}$-torsor and $G$-bundles.
\nomenclature{$\cG^{std}$}{Sheaf of groups associated to the constant group scheme.}

Let $P\subset G$ be a parabolic subgroup. Let $L^+_{P}G = \{\ga \in G[[z]] \ | \ga(0) \in P\}$. Construct a sheaf of groups $\mc{G}^P$ on $\ec \C[[z]] = \{(z), (0)\}$ by $\mc{G}^P(\{(z), (0)\}) = L^+_{P}G$ and $\mc{G}^P(\{0\}) = G((z))$. Given a smooth curve $C$ and a point $p$ we notice that $\cG^{std}|_{C-p}$ and $\mc{G}^P$ agree over $\ec \C((z))$ and thus define a sheaf of group which we also denote $\mc{G}^P$.  Clearly we can iterate over $(x_i) = x_1,\dotsc, x_m \in C$ with parabolics $(P_i) = P_1, \dotsc, P_m$.  Call the resulting sheaf of groups $\cG^{(x_i),(P_i)}$.  Then $\cG^{(x_i),(P_i)}$-torsors are exactly quasi parabolic bundles: $G$-bundles on $C$ with reduction of structure group to $P_i$ at $x_i$.
\nomenclature{$L^+_{P}G$}{Parahoric subgroup of $LG$ associated to a parabolic of $G$}

In the examples mentioned thus far all the $\cG$ have the property that $\cG(\hat \cO_x) \subset \cG^{std}(\hat \cO_x) = G[[z]]$ where $\hat \cO_x$ is the completion of the local ring with respect to the maximal ideal. This inclusion is the crucial property that allows a dictionary between torsors and principal bundles.  This inclusion fails in general and consequently torsors are more general objects that bundles.
\nomenclature{$\cG(\hat \cO_x)$}{completed stalk of a sheaf of groups}

We can construct non bundle torsors by considering parahoric subgroups; these generalize $L^+_{P}G$. A {\it parahoric subgroup} of $\Gsd$ is any subgroup $\mc{P} \subset \Gsd$ that is conjugate in $\Gsd$ to one of the groups
\[
\mc{P}^\ltimes_\eta = \{\ga \in \Gsd| \lim_{s \to 0} \eta(s)\ga\eta(s)^{-1} \mbox{ exists } \}
\]
where $\eta$ is a co-character such that the composition $\Cs \xrightarrow{\eta} \Tsd = \Cs \x T \xrightarrow{p_1} \Cs$ is given by a nonzero integer.  We use the notation $\mc{P}_\eta$ to denote the quotient $\mc{P}^\ltimes_\eta/\Cs = \mc{P}^\ltimes_\eta \cap LG$ and define {\it parahoric subgroups of $LG$} to be any subgroup conjugate to a $\mc{P}_\eta$.  The groups $\mc{P}^\ltimes_\eta$ come with a natural Levi decomposition: $L_\eta = \{\ga \in \mc{P}^\ltimes_\eta | \lim_{s \to 0} \eta(s)\ga\eta(s)^{-1} = \ga \}$ and $\hat U_\eta = \{\ga \in \mc{P}^\ltimes_\eta | \lim_{s \to 0} \eta(s)\ga\eta(s)^{-1} = id \}$.

Given a set $(\mc{P}_i)$ of parahoric subgroups we can analogously construct a sheaf of groups $\cG^{(x_i),(\mc{P}_i)}$ and consider its torsors.  Such torsors, also known as quasi parahoric bundles, seemed to be first discussed by Teleman, see \cite[sect. 9]{T2}.  They have recently received more attention by Heinloth \cite{Hein} and by Balaji and Seshadri \cite{Balaji}.

\subsection{$\mc{G}$-torsors on nodal curves}\label{ss:Gnode}

Associate to $j \in \{0,\dotsc, r\}$ the vertex $\eta_j \in Al_0$ satisfying $\a_j(\eta_j) >0$ and all other $\a_i(\eta_j) =0$. Set $\mc{P}_j = \mc{P}_{\eta_j}$ with Levi decomposition $L_jU_j := L_\eta U_\eta$; we call these the {\it standard maximal parahoric subgroups}.   

Now consider again $D = \ec \C[[x,y]]/xy = \{(x,y), (y), (x)\}$ and $D^* = D - \{(x,y)\}$.  Define a sheaf of groups $\cG^j$ by $\cG^j(D) = \Delta(L_j) \ltimes (\hat U_j \x \hat U_j) $ and $\cG^j|_{D^*} = \cG^{std}$; the restriction map $\cG^j(D)\to \cG^j(D^*)$ is given by $\Delta(L_j) \ltimes (U_j \x U_j) \subset \mc{P}_j \x \mc{P}_j \subset G((z)) \x G((z)) \cong G((x)) \x G((y))$. Similarly as in the smooth curve case we can extend $\cG^j$ to a sheaf groups on a nodal curve $C$ with a unique node $x$ with the property that $\cG^j|_{C -x} = \cG^{std}$.  We can repeat this construction with $G_{ad}$ in place of $G$ and take connected components of the identity throughout; call the resulting sheaf of groups $\cG^j_{ad}$.

Recall from proposition \ref{p:Gorb} that $\Delta(L_j) \subset H_j$; here $L_j$ is a Levi factor of a {\it maximal} parahoric subgroup of $LG$ and consequently $L_j$ is a semi simple group.  Let $Z_j = Z(L_j)/Z(G)$; for $SL_n$ these groups are trivial but not in general. An example for $G = SO_5$ is given at the end of this section.

For $Z_j$ as above let $C^j$ be a curve as in before proposition \ref{p:O0} with the addendum that the nodal point $x$ is replaced by a stacky point $[pt/Z_j]$. Let $C$ be the corse moduli space of $C^j$; i.e. where $x$ is replaced an ordinary non stacky point. Let $\tilde C$ be the  normalization.  Define $\tilde C^j$ to be $\tilde C$ with stacky points $[pt/Z_j]$ introduced at the pre images $y,z$ of the node.  We fix a map $\tilde C^j \to C^j$ by identifying $y,z$ via the trivial automorphism.

Define a stack $Bun^0_{\mc{G}^j_{ad},C^*}C^j$ on $\C$-algebras by
\[
Bun^0_{\mc{G}^j_{ad},C^*}C^j(R) =\bigg \l  \mc{G}^j_{ad}-\mbox{torsors} \begin{array}{c}
     \mc{F} \\ 
    \downarrow\\ 
     C^j_R \\ 
  \end{array}, \tau\colon \mc{F}|_{C^*_R} \widetilde{\to}  \mc{G}^j_{ad}|_{C^*_R}
\bigg \r
\]

\begin{prop}\label{p:Oj}
Let $Z_j$ be the finite group described above. Let $C^j$ be a nodal curve with a unique stacky node $x$ such that $C-x$ is affine and $x \cong [pt/Z^j]$. The stack $Bun^0_{\mc{G}^j_{ad},C^*}C^j$ is represented by the orbit $\bO_j \subset \Xap$.
\end{prop}
\begin{proof}
If $x$ had no stackyness then we can use the local model $\ec \hat \cO_x = \ec \C[[y,z]]/yz$ and isomorphisms of torsors over $\ec \hat \cO_x$ is by definition an element of $\mc{G}^j_{ad}(\hat \cO_x) = \Delta(L_j/Z(G)) \ltimes (\hat U_j \x \hat U_j)$.  

When working on the stacky curve we can lift the automorphisms of $x = [pt/Z_j]$ and so isomorphisms of torsors over $\ec \hat \cO_x$ is identified with $Z_j \x Z_j \cdot \Delta(L_j/Z(G)) \ltimes (\hat U_j \x \hat U_j)$. Further by \cite{Hein}, any $\cG$-torsor can be trivialized on $C-x$ so we are formally in the same situation as proposition \ref{p:O0} and we can use the same argument to show
\begin{equation}\label{rhs}
Bun^0_{\mc{G}^j_{ad},C^*}C^j \cong \frac{L G_{ad} \x L G_{ad}}{Z_j \x Z_j \cdot \Delta(L_j/Z(G)) \ltimes (\hat U_j \x \hat U_j)} = \frac{LG \x LG}{Z(L_j) \x Z(L_j)\cdot \Delta(L_j) \ltimes (\hat U_j \x \hat U_j)}
\end{equation}
So it suffices to show the right hand side is isomorphic to $\bO_j$. From \ref{p:Gorb} we have
\begin{align*}
\bO_j = \frac{\Gsdp \x \Gsdp}{Z(L_j) \x Z(L_j)\Delta(L_j) \ltimes (U_j \x U_j^-) } \ \ \ \ \ \ \ Z(L_j) = \bigcap_{i \ne j}\ker \a_i
\end{align*}
Choose a co-character $\zeta$ such that $\zeta(t) \in \cap_{i \ne j}\ker \a_i \cap T = \cap_{i \ne j,0} \ker \a_i$ and $\a_j(\zeta(t)) = t^{\l \zeta, \a \r} \ne 1$; this is always possible if $j \ne 0$. Let $\theta$ be the longest root of $G$ and write $\theta = \sum_{i=1}^r n_i \a_i$ with $n_i>0$.  Then
\[
(u,\zeta(t)) \xrightarrow{\a_0} u \prod_{i \ge 1} t^{n_i \l\zeta, \a_i\r} = u t^{n_j \l \zeta, \a_j\r} = u t^m
\]
with $m \ne 0$ and thus the co-character $u \xrightarrow{\zeta'} (u, \zeta(u)^{-1/m})$
satisfies $\zeta'(u) \x \zeta'(u) \subset T(J)$ and for dimensional reasons we have equality.  This means the class of $(u,\ga), (u',\ga') \in \Gsdp \x \Gsdp$ in $\bO_j$ has a unique representative of the form $(1,h),(1,h')$ where the $\ga$'s and $h$'s differ by multiplication by an element of $Z(L_j)$.  Consequently, 
\[
\bO_j = \frac{\LpG \x \LpG}{Z(L_j) \x Z(L_j) \cdot \Delta(L_j) \ltimes (U_j \x U_j^-) }
\]
and as in proposition \ref{p:O0} we establish that $\bO_j$ is isomorphic to the right hand side of equation \eqref{rhs}
\end{proof}

Let $\tilde C^j$ be the normalization of $C^j$ as described before proposition \ref{p:Oj} and let $y,z$ be the preimages of $x$.  And let $\mc{P}_j$ be the parahoric subgroup of $LG$ described at the start of subsection \ref{ss:Gnode}. Using the construction outlined after proposition \ref{p:O0} we associate to the data $(y,z),(\mc{P}_j,\mc{P}_j)$ a sheaf of groups $\mc{G}^{(y,z),(\mc{P}_j,\mc{P}_j)}$ on $\tilde C^j$ and torsors for $\mc{G}^{(y,z),(\mc{P}_j,\mc{P}_j)}$ are quasi parahoric bundles.
\begin{cor}
$Bun^0_{\mc{G}^j_{ad},C^*}C^j$ can be identified with a quasi parahoric bundle $\widetilde{\mc{F}}$ on $\tilde C^j$, with the same trivialization $\tau$ and additionally a reduction of the $\mc{P}_j \x \mc{P}_j$ torsor $\widetilde{\mc{F}}(\hat \cO_y) \x \widetilde{\mc{F}}(\hat \cO_z)$ to $\mc{G}^j(\hat \cO_x) \subset \mc{P}_j \x \mc{P}_j$.

\end{cor}
\begin{proof}
Given an $(\mc{F}, \tau) \in Bun^0_{\mc{G}^j_{ad},C^*}C^j(R)$ we associate to this data a sheaf of groups $\widetilde{\mc{G}}^j$ on $\tilde C$, a torsor $\widetilde{\mc{F}}$ and a section as in the statement of the proposition.

Fix a local model for $x \in C$ where the {\it point} $x$ corresponds to $(y,z) \subset \ec \C[[y,z]]/y z$ where by abuse of notation we consider $y,z$ as local coordinates near the {\it points} $y,z \in \tilde C$. Then $\widetilde{\mc{G}}^j = \mc{G}^j$ on $\tilde C - y - z$ and we define $\widetilde{\mc{G}}^j(\hat \cO_y)$ to be the image of the composition $\mc{G}^j(\hat \cO_x)\subset G((y)) \x G((z)) \to G((y))$ and similarly for $\widetilde{\mc{G}}^j(\hat \cO_z)$.

To define $\widetilde{\mc{F}}$ note that a trivialization of $\mc{F}$ over $\ec \hat \cO_x$ induces an element in $\mc{P}^j = \widetilde{\mc{G}}^j(\hat \cO_y)$ and comparing it with the trivialization $\tau$ defines `transition function' $\in G((y))$ which we use to define $\widetilde{\mc{F}}$ over $C^*\cup y$; this is possible again by the uniformization of torsors over a smooth curve proved in \cite{Hein}. Extending to $C^* \cup y \cup z$ is handled similarly.   

To define the reduction note that an element $e \in \mc{F}(\hat \cO_x)$ defines a point $(e,e) \in \mc{F}(\hat \cO_y) \x \mc{F}(\hat \cO_z)$ and another choice $e'$ differs from $e$ by an element of $\mc{G}^j(\hat \cO_x)$, so we have a well defined section of $ \widetilde{\mc{F}}(\hat \cO_y) \x \widetilde{\mc{F}}(\hat \cO_z)/\mc{G}^j(\hat \cO_x)$.

We show now given data on $\tilde C$ we can get and element of $\bO_j$ and composing with $\bO_j \widetilde{\to}Bun^0_{\mc{G}^j_{ad},C^*}C^j$ gives the inverse to the above construction.  Again extending to $R \to R'$ we assume $\widetilde{\mc{F}}$ has a trivialization over $\ec (\hat \cO_y \x \hat \cO_z)$ comparing with the trivialization $\tau$ we get elements $(\ga_1, \ga_2) \in G((y)) \x G((z))$.  We can extend in two different ways to $R'' = R'\ox_R R'$ giving us two sets of elements $(\ga_1', \ga_2')$ and $(\ga_1'',\ga_2'')$ in $G(R''((y))) \x G(R''((z)))$.  These elements are related by an isomorphism of $\widetilde{\mc{F}}$ over $\ec (\hat \cO_y \x \hat \cO_z)$ that preserves the section; such isomorphism come from multiplication by $\Delta(L_j) \x (\hat U_j \x \hat U_j)$. But there is also an additional factor of $Z(L_j) \x Z(L_j)$ that comes from lifting the automorphisms of the points $y,z$. Arguing as in proposition \ref{p:O0} we get an element of $\bO_j(R)$ as desired.
\end{proof}

\begin{rmk}\label{rmk:alg}
The question of the modular interpretation of the remaining, higher codimensional orbits requires more care.  There are some results in this direction when one restricts to the divisor $\ol{\bO_0}$.  In this case higher codimensional orbits have been interpreted as torsion free sheaves for $GL_n$ and $Sp_n$.  The other approach is to consider bundles on modifications of nodal curves as in \cite{Tolland,K1,K2}.  In the analytic setting one can nevertheless fit together all the orbits of $\Xa$ into a complex analytic space that serves as a completion of bundles over nodal curves in families. The appropriate algebraic analogue is work in progress.
\end{rmk}

\begin{rmk}
The space $Bun^0_{\mc{G}^j_{ad},C^*}C^j$ is a `rigidified' moduli space and when one forgets the rigidifications one gets a space much closer to $Bun_G(C)$. Indeed, the $0$ and $ad$ decorations can be removed by working with $\mc{X}^{aff}_{poly}$ instead $\Xap$.  There is an action of $G(C^*)$ on $Bun^0_{\mc{G}^j_{ad},C^*}C^j$ which changes the trivialization over $C^*$. Quotienting by this action gives a moduli space which we can denote $Bun_{\mc{G}^j}C^j$ and in turn we can pass from $\mc{G}^j$ to $G$ by working with certain equivariant $G$-bundles.  The details of this construction will be carried out in a follow up paper.
\end{rmk}

Even without a modular description of the other orbits, the previous proposition already tells us something interesting.  Given that bundles on a nodal curve are equivalent to bundles on the normalization together with a `transition function'$\in G$ at the node, it is a natural first guess to try to complete $Bun_G(C)$ simply by compactifying $G$.  However the previous proposition shows that this is not sufficient; namely, in families a principal bundle may develop parahoric structure at the node.  Figure \ref{pic:bunG} is an illustrates this; compactifying $G$ only tells you about the divisor $D_0$.
\begin{figure*}[htm]
\centering
\includegraphics[scale=0.45]{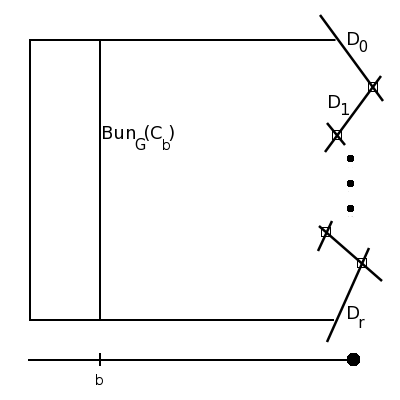}
\caption{Schematic of $\ol{Bun_G(C/B)}$ over the base $B$.  The general fiber is just $Bun_G(C_b)$ for the smooth curve $C_b$.  Over the special fiber of $B$ we get a divisor with normal crossings and $r+1$ components.  The divisor $D_i$ has a dense open subset, the complement of the squares, which corresponds to $Orb(\{i\})$.}
\label{pic:bunG}
\end{figure*}

We now give a couple of example of parahoric subgroups.  We treat first the case $G = SL_2$.  Set $\cO = \C[[z]]$. The standard parahoric subgroups of $LG$ are
\begin{align*}
\mc{P}_0 &= G(\cO) = G[[z]]\\
\mc{P}_1 &:= \left\{\ga(z) = \abcd{\cO}{z^{-1} \cO}{z \cO}{\cO} | \det \ga(z) = 1\right\}
\end{align*}
The first corresponds to the co-character $u \mapsto (u,diag(1,1)) \in \Tsd$ and the second corresponds to $(u,diag(\sqrt{u},\frac{1}{\sqrt{u}}) )$.

In the case of $SL_r$, all the maximal parahoric bundles are conjugate to $G[[z]]$ by outer automorphisms which can either be interpreted as `fractional' loops in $G$ or as honest loops in $GL_r$ .  For example
\[
g_1\mc{P}_1g_1^{-1} = g_2\mc{P}_1g_2^{-1} = \mc{P}_0 \ \ \ \mbox{with } g_1 = \abcd{1}{0}{0}{z^{-1}}, \ \ \ g_2 = \abcd{z^{1/2}}{0}{0}{z^{-1/2}}
\]

As a result all $\cG$-torsors with $\mc{P}_i$ structure always have interpretations in terms of vector bundles.  In the case at hand if $E$ is the trivial rank $2$ vector bundle over $D = \ec \C[[z]]$ and $s_1,s_2$ are two non vanishing sections then any other such sections can be obtained as $\ga.s1, \ga.s2$ with $\ga \in \mc{P}_0$.  On the other hand if we require $s_1$ to have a pole at the closed point and $s_2$ to be non vanishing at the closed point then any other such sections can be obtained as $\ga.s1,\ga.s2$ with $\ga \in \mc{P}_1$.  The latter case corresponds to the vector bundle $\cO_D \oplus \cO_D(1)$ over $D$.  In other words, if we have a smooth curve $C$ and a point $p$ then $G(C-p)\backslash LG/\mc{P}_0$ is the moduli space of rank 2 vector bundles with trivial determinant while $G(C-p)\backslash LG/\mc{P}_1$ corresponds to those vector bundles with determinant $\cO_C(p)$.

We finish with one final example of a parahoric subgroup of $SO_5$ which shows, unlike the $SL_r$ case, the parahoric structures that appear cannot always be interpreted as twists of the standard structure $G[[z]]$.  It will be enough to work with the Lie algebra which we present as
\[
\mathfrak{so}_5 =  \left\{\left[ \begin{array}{ccccc}
    a_1 & a_2 & 0 & b & -h1 \\ 
    a_3 & a_4 & -b & 0 & -h2 \\ 
    0 & c & -a_1 & -a_3 & -g1 \\ 
    -c & 0 & -a_2 & -a_4 & -g_2 \\ 
    g_1 & g_2 & h_1 & h_2 & 0 \\ 
  \end{array}\right] | a_i,b,c,g_i,h_i \in \C \right\}
\]
the Lie algebra of a maximal torus is given by diagonal matrices and more generally we can pick out the root spaces
\begin{align*}
(X_1,Y_1) &\leftrightarrow (a_2,a_3)\\
(X_2,Y_2) &\leftrightarrow (\sqrt 2 h_2, - \sqrt 2 g_2)\\
(X_3,Y_3) &\leftrightarrow (\sqrt 2 h_1, - \sqrt 2 g_1)\\
(X_4,Y_4) &\leftrightarrow (b,c)
\end{align*}
For example
\begin{align*}
X_1 = \left[  \begin{array}{ccccc}
      & 1  &   &   &   \\ 
      &   &   &   &   \\ 
      &   &   &   &   \\ 
      &   & -1  &   &   \\ 
      &   &   &   &   \\ 
  \end{array}\right] \ \ \ \ Y_1  = \left[  \begin{array}{ccccc}
      &   &   &   &   \\ 
   1 &   &   &   &   \\ 
      &   &   &  -1 &   \\ 
      &   &   &   &   \\ 
      &   &   &   &   \\ 
  \end{array}\right]
\end{align*}
The maximal parahoric subgroups are given by the vertices of the Weyl alcove $\subset \mft$.  Identify $\mft = \C^2$ using the basis $H_1 = [X_1,Y_1]$ and $H_2 = [X_2,Y_2]$.  Then the Weyl alcove is a triangle with supporting hyperplanes
\begin{align*}
y = x, \ \ \ \  y = x/2, \ \ \ \ y = 1/2
\end{align*}
The vertices are $v_0 = (0,0), v_1 = (1,1/2)$ and $v_2 = (1,1/2)$ and they give rise to three maximal parahorics $\mc{P}_0, \mc{P}_1,\mc{P}_2$ .  We have $\mc{P}_0 = G[[z]]$ and $\mc{P}_0 \cong \mc{P}_1$ via an outer automorphism of $LSO(5)$.  In fact the outter automorphisms of $LG$ come from automorphisms of the corresponding affine Dynkin diagram.  For $L SO(5)$ the diagram looks like $0 \rightrightarrows 2 \leftleftarrows 1$ and the automorphism in question switches $0,1$.

The parahoric $\mc{P}_2$ has a Levi factor isomorphic to $SL_2 \x SL_2$ the Lie algebra of this levi is given by 
\[
\mathfrak{sl}_2 \x \mathfrak{sl_2} \cong \abcd{H_4}{Y_4 \otimes z}{X_4 \otimes z^{-1}}{-H_4} \x \abcd{H_1}{X_1}{Y_1}{-H_1} 
\] 
In this case the groups $Z_2$ referred to in proposition \ref{p:Oj} is $\Z/2$. The non isomorphic Levi subgroups occur because of the different length roots in $\mathfrak{so}_5$.

\begin{thenomenclature} 

 \nomgroup{A}

  \item [{$(\Sigma, \beta \colon L \to N)$}]\begingroup Stacky fan.\nomeqref {5}
		\nompageref{8}
  \item [{$[Y(\Sigma)/Z(\beta)]$}]\begingroup Toric stack associated to a stacky fan.\nomeqref {5}
		\nompageref{8}
  \item [{$\a_0$}]\begingroup Affine root of $\Ga$\nomeqref {7}
		\nompageref{11}
  \item [{$\a_1, \dotsc, \a_r$}]\begingroup Roots of a semi simple group\nomeqref {3}
		\nompageref{5}
  \item [{$\beta^*$}]\begingroup Dual of the map $\beta$ in a stacky fan\nomeqref {5}
		\nompageref{8}
  \item [{$\bO_j$}]\begingroup Orbit in $\Xa$ associated to the subset $\{j\}$\nomeqref {12}
		\nompageref{25}
  \item [{$\cG$}]\begingroup sheaf of groups on a family of curves\nomeqref {12}
		\nompageref{26}
  \item [{$\cG(\hat \cO_x)$}]\begingroup completed stalk of a sheaf of groups\nomeqref {12}
		\nompageref{26}
  \item [{$\cG^{std}$}]\begingroup Sheaf of groups associated to the constant group scheme.\nomeqref {12}
		\nompageref{26}
  \item [{$\Ga$}]\begingroup Affine Kac-Moody group.\nomeqref {7}
		\nompageref{11}
  \item [{$\Ga_{ad} = \Gsd/Z(G)$}]\begingroup  Adjoint group of $\Ga$\nomeqref {10}
		\nompageref{14}
  \item [{$\Gap$}]\begingroup Polynomial affine Kac-Moody group associated to $G$\nomeqref {8}
		\nompageref{12}
  \item [{$\Gsd$}]\begingroup Semi direct loop group\nomeqref {7}
		\nompageref{11}
  \item [{$\l \mu ,\eta \r$}]\begingroup Paring between characters and co-characters\nomeqref {3}
		\nompageref{5}
  \item [{$\LpG$}]\begingroup Polynomial loop group\nomeqref {8}
		\nompageref{12}
  \item [{$\mft^\ltimes_\R$}]\begingroup $\R$ Lie algebra of $\Tsd$\nomeqref {10}
		\nompageref{14}
  \item [{$\ol{\LTsd}$}]\begingroup Embedding of the $\LTsd$\nomeqref {11}
		\nompageref{20}
  \item [{$\ol{\Tad}$}]\begingroup  Closure of $\Tad$ in $X$\nomeqref {4}
		\nompageref{5}
  \item [{$\ol{\Tsad}$}]\begingroup  Closure of $\Tsad$ in $\Xa$\nomeqref {10}
		\nompageref{14}
  \item [{$\ol{\Tsd_Q}$}]\begingroup toric variety inside an embedding for the group $\tilde Q$\nomeqref {12}
		\nompageref{22}
  \item [{$\oTad$}]\begingroup  Affine toric variety, open cell of $\ol{\Tad}$\nomeqref {4}
		\nompageref{5}
  \item [{$\oTsad$}]\begingroup  Affine toric variety, open cell of $\ol{\Tsad}$\nomeqref {10}
		\nompageref{14}
  \item [{$\Ta$}]\begingroup maximal torus of $\Ga$\nomeqref {7}
		\nompageref{11}
  \item [{$\Tad$}]\begingroup  Maximal torus for $G_{ad}$\nomeqref {4}
		\nompageref{5}
  \item [{$\Taf$}]\begingroup central extension of $\Cs \ltimes LT$\nomeqref {11}
		\nompageref{20}
  \item [{$\tilde Q$}]\begingroup The group $\tvt$ with group structure determined by $Q$\nomeqref {11}
		\nompageref{21}
  \item [{$\tilde Q_{ad}$}]\begingroup adjoint group of $\tilde Q$\nomeqref {11}
		\nompageref{21}
  \item [{$\tlt0$}]\begingroup subgroup of $\Taf$\nomeqref {11}
		\nompageref{20}
  \item [{$\Tsad$}]\begingroup  Maximal torus for $\Ga_{ad}$\nomeqref {10}
		\nompageref{14}
  \item [{$\Tsd$}]\begingroup Maximal torus of $\Gsd$\nomeqref {7}
		\nompageref{11}
  \item [{$\tvt$}]\begingroup subgroup of $\Taf$\nomeqref {11}
		\nompageref{20}
  \item [{$\Wa$}]\begingroup Affine Weyl group\nomeqref {7}
		\nompageref{12}
  \item [{$\Xa$}]\begingroup Wonderful embedding of $\Gsd/Z(G)$\nomeqref {10}
		\nompageref{14}
  \item [{$\Xa_0$}]\begingroup  Open cell of $\Xa$\nomeqref {10}
		\nompageref{14}
  \item [{$\Xap$}]\begingroup Polynomial embedding\nomeqref {10}
		\nompageref{14}
  \item [{$\Xas$}]\begingroup Smooth embedding\nomeqref {10}
		\nompageref{14}
  \item [{$Al_0$}]\begingroup Positive Weyl alcove\nomeqref {8}
		\nompageref{12}
  \item [{$C_\Delta$}]\begingroup Cone depending on the fundamental weights of a semisimple group\nomeqref {6}
		\nompageref{9}
  \item [{$G_{ad}$}]\begingroup  Adjoint group of $G$\nomeqref {4}
		\nompageref{5}
  \item [{$H_j$}]\begingroup Subgroup associated to the orbit $\bO_j$\nomeqref {12}
		\nompageref{25}
  \item [{$L^+_{P}G$}]\begingroup Parahoric subgroup of $LG$ associated to a parabolic of $G$\nomeqref {12}
		\nompageref{26}
  \item [{$L^+G$,$L^-G$,$\cB$, $\cB^-$, $\cU$, $\cU^-$}]\begingroup Subgroups of the loop group\nomeqref {7}
		\nompageref{11}
  \item [{$P(p)$}]\begingroup convex hull of a finite number of points in a vector space.\nomeqref {12}
		\nompageref{22}
  \item [{$T$, $B$,$B^-$,$U$,$U^-$}]\begingroup Subgroups of a semisimple group\nomeqref {3}
		\nompageref{5}
  \item [{$u_1, \dotsc, u_r$}]\begingroup Ray generators for the cone given by the Weyl chamber.\nomeqref {5}
		\nompageref{9}
  \item [{$X$}]\begingroup Wonderful compactification of $G_{ad}$\nomeqref {4}
		\nompageref{5}
  \item [{$X_0$}]\begingroup  Open cell of $X$\nomeqref {4}
		\nompageref{5}
  \item [{$Y(C_\Delta)$}]\begingroup Quasi-projective $G \x (\Cs)^r$-embedding\nomeqref {6}
		\nompageref{9}
  \item [{$Z(\beta)$}]\begingroup Subgroup of a torus associated to a stacky fan.\nomeqref {5}
		\nompageref{8}
  \item [{$Z(G)$}]\begingroup Center of a group $G$\nomeqref {4}
		\nompageref{5}
  \item [{$Z_Q$}]\begingroup Subgroup of torus assoicated to a quadratic form $Q$.\nomeqref {11}
		\nompageref{21}
  \item [{HW}]\begingroup highest weight\nomeqref {3}\nompageref{5}
  \item [{HWR}]\begingroup Highest weight representation\nomeqref {3}
		\nompageref{5}
  \item [{PER}]\begingroup positive energy representation\nomeqref {8}
		\nompageref{13}

\end{thenomenclature}

\bibliographystyle{plain} 
\bibliography{WELG}

\def\cprime{$'$}
\begin{thebibliography}{10}

\bibitem{Bea2}
Arnaud Beauville.
\newblock Conformal blocks, fusion rules and the {V}erlinde formula.
\newblock In {\em Proceedings of the {H}irzebruch 65 {C}onference on
  {A}lgebraic {G}eometry ({R}amat {G}an, 1993)}, volume~9 of {\em Israel Math.
  Conf. Proc.}, pages 75--96, Ramat Gan, 1996. Bar-Ilan Univ.

\bibitem{Bea}
Arnaud Beauville and Yves Laszlo.
\newblock Conformal blocks and generalized theta functions.
\newblock {\em Comm. Math. Phys.}, 164(2):385--419, 1994.

\bibitem{Bhosle}
Usha~N. Bhosle.
\newblock Principal {$G$}-bundles on nodal curves.
\newblock {\em Proc. Indian Acad. Sci. Math. Sci.}, 111(3):271--291, 2001.

\bibitem{Borel}
Armand Borel.
\newblock {\em Linear algebraic groups. Second edition. Graduate Texts in
  Mathematics, 126.}
\newblock Springer-Verlag, New York, NY, 1991.

\bibitem{Brav}
Alexander Braverman and Michael Finkelberg.
\newblock Pursuing the double affine {G}rassmannian. {I}. {T}ransversal slices
  via instantons on {$A_k$}-singularities.
\newblock {\em Duke Math. J.}, 152(2):175--206, 2010.

\bibitem{Brion}
Brion and Kumar.
\newblock {\em The Frobenius Splitting Method in Algebraic Geometry, Progress
  in Mathematics, 231}.
\newblock Birkhaeuser, Boston, MA, 2005.

\bibitem{Brion2}
Michel Brion.
\newblock The total coordinate ring of a wonderful variety.
\newblock {\em J. Algebra}, 313(1):61--99, 2007.

\bibitem{DC&P}
C.~De~Concini and C.~Procesi.
\newblock Complete symmetric varieties.
\newblock In {\em Invariant theory ({M}ontecatini, 1982)}, volume 996 of {\em
  Lecture Notes in Math.}, pages 1--44. Springer, Berlin, 1983.

\bibitem{Drin}
V.~G. Drinfel{\cprime}d and Carlos Simpson.
\newblock {$B$}-structures on {$G$}-bundles and local triviality.
\newblock {\em Math. Res. Lett.}, 2(6):823--829, 1995.

\bibitem{Fa2}
Gerd Faltings.
\newblock A proof for the {V}erlinde formula.
\newblock {\em J. Algebraic Geom.}, 3(2):347--374, 1994.

\bibitem{Fa1}
Gerd Faltings.
\newblock Moduli-stacks for bundles on semistable curves.
\newblock {\em Math. Ann.}, 304(3):489--515, 1996.

\bibitem{Fuchs}
J{\"u}rgen Fuchs.
\newblock {\em Affine {L}ie algebras and quantum groups}.
\newblock Cambridge Monographs on Mathematical Physics. Cambridge University
  Press, Cambridge, 1992.
\newblock An introduction, with applications in conformal field theory.

\bibitem{Fulton}
William Fulton.
\newblock {\em Intersection theory}, volume~2 of {\em Ergebnisse der Mathematik
  und ihrer Grenzgebiete. 3. Folge. A Series of Modern Surveys in Mathematics
  [Results in Mathematics and Related Areas. 3rd Series. A Series of Modern
  Surveys in Mathematics]}.
\newblock Springer-Verlag, Berlin, second edition, 1998.

\bibitem{Anton}
A.~Geraschenko and M.~Sattriano.
\newblock Toric stacks i: The theory of stacky fans.
\newblock {\em arxiv 1107.1906}.

\bibitem{Matt}
A.~Geraschenko and M.~Sattriano.
\newblock Torus quotients as global quotients by finite groups.
\newblock {\em arxiv 1201.4807}.

\bibitem{Hein}
Jochen Heinloth.
\newblock Uniformization of {$\mathcal{G}$}-bundles.
\newblock {\em Math. Ann.}, 347(3):499--528, 2010.

\bibitem{Humph}
J.~Humphreys.
\newblock {\em Linear Algebraic Groups. Graduate Texts in Mathematics, No. 21.}
\newblock Springer-Verlag, New York-Heidelberg, 1975.

\bibitem{K1}
Ivan Kausz.
\newblock A modular compactification of the general linear group.
\newblock {\em Doc. Math.}, 5:553--594 (electronic), 2000.

\bibitem{K2}
Ivan Kausz.
\newblock A {G}ieseker type degeneration of moduli stacks of vector bundles on
  curves.
\newblock {\em Trans. Amer. Math. Soc.}, 357(12):4897--4955 (electronic), 2005.

\bibitem{Kumar}
Kumar.
\newblock Kac-moody groups, their flag varieties \& representation theory.
\newblock 2002.

\bibitem{Sorger}
Yves Laszlo and Christoph Sorger.
\newblock The line bundles on the moduli of parabolic {$G$}-bundles over curves
  and their sections.
\newblock {\em Ann. Sci. \'Ecole Norm. Sup. (4)}, 30(4):499--525, 1997.

\bibitem{Lamp}
L{\'a}szl{\'o} Lempert.
\newblock Vanishing cohomology for holomorphic vector bundles in a {B}anach
  setting.
\newblock {\em Asian J. Math.}, 8(1):65--85, 2004.

\bibitem{Martens}
Johan Martens and Michael Thaddeus.
\newblock Compactifications of reductive groups as moduli stacks of bundles.
\newblock {\em arxiv 1105.4830}.

\bibitem{loop}
Presseley and Segal.
\newblock {\em Loop Groups}.
\newblock Oxford Science Publications. The Clarendon Press, Oxford University
  Press, New York, NY, 1986.

\bibitem{Vakil}
Mike Roth and Ravi Vakil.
\newblock The affine stratification number and the moduli space of curves.
\newblock In {\em Algebraic structures and moduli spaces}, volume~38 of {\em
  CRM Proc. Lecture Notes}, pages 213--227. Amer. Math. Soc., Providence, RI,
  2004.

\bibitem{Segal}
Graeme Segal.
\newblock Unitary representations of some infinite-dimensional groups.
\newblock {\em Comm. Math. Phys.}, 80(3):301--342, 1981.

\bibitem{T2}
Constantin Teleman.
\newblock The quantization conjecture revisited.
\newblock {\em Ann. of Math. (2)}, 152(1):1--43, 2000.

\bibitem{Tolland}
AJ~Tolland, E~Frenkel, and C~Teleman.
\newblock Gromov-witten gauge theory i.
\newblock {\em arxiv 0904.4834}.

\bibitem{Voronoi}
Delaunay Triangulations.
\newblock Chapter 8 dirichlet - voronoi diagrams and.
\newblock \url{http://www.cis.upenn.edu/~cis610/convex8.pdf}.

\bibitem{Balaji}
C.S.~Seshadri V.~Balaji.
\newblock Moduli of parahoric $\mathcal{G}$--torsors on a compact riemann
  surface.
\newblock {\em arxiv 1009.3485}.

\bibitem{Vinberg}
E.~B. Vinberg.
\newblock On reductive algebraic semigroups.
\newblock In {\em Lie groups and {L}ie algebras: {E}. {B}. {D}ynkin's
  {S}eminar}, volume 169 of {\em Amer. Math. Soc. Transl. Ser. 2}, pages
  145--182. Amer. Math. Soc., Providence, RI, 1995.

\end{thebibliography}

\end{document}